\documentclass[a4paper,USenglish,cleveref, autoref, thm-restate]{lipics-v2021}
\hideLIPIcs

\usepackage{xcolor}
\usepackage{placeins}
\usepackage{xspace}
\usepackage{csquotes}
\usepackage{bm}
\usepackage{complexity}
\usepackage[sort]{cite}

\crefname{observation}{Observation}{Observations}

\newcommand{\refcommand}[1]{\textnormal{\sffamily[\cref{#1}]}}
\DeclareCaptionLabelFormat{bold}{\textbf{\sffamily(#2)}}
\newcommand{\obs}{\mathcal{G}_d}
\definecolor{signalblue1}{RGB}{58,118,240}
\definecolor{KITred}{rgb}{0.627,0.117,0.156}
\definecolor{KITorange}{rgb}{0.862,0.627,0.117}
\definecolor{KITgreen}{rgb}{0,0.588,0.509}
\DeclareTextFontCommand{\emph}{\color{signalblue1}\em}
\DeclareMathOperator{\tw}{tw}

\NewDocumentCommand{\calP}{}{\ensuremath{\mathcal{P}}\xspace}
\NewDocumentCommand{\calL}{}{\ensuremath{\mathcal{L}}\xspace}
\NewDocumentCommand{\calG}{}{\ensuremath{\mathcal{G}}\xspace}
\let\oldFPT\FPT
\renewcommand{\FPT}{\oldFPT\xspace}

\title{Separating Geodesic Structure and Product Structure}

\author{Laura Merker}{Karlsruhe Institute of Technology, Germany}{laura.merker2@kit.edu}{https://orcid.org/0000-0003-1961-4531}{}
\author{Lena Scherzer}{University of Hamburg, Germany}{lena.scherzer@uni-hamburg.de}{https://orcid.org/0009-0001-1524-7378}{}
\author{Samuel Schneider}{Karlsruhe Institute of Technology, Germany}{samuel.schneider@kit.edu}{https://orcid.org/0009-0002-9680-4048}{}

\authorrunning{L. Merker, L. Scherzer, and S. Schneider}

\Copyright{Laura Merker, Lena Scherzer, and Samuel Schneider}


\category{}

\relatedversion{ An extended abstract of this paper appears in the proceedings of the \textit{34th Annual European Symposium on Algorithms} (ESA 2026).}

\nolinenumbers

\begin{document}

\maketitle

\begin{abstract}
    The geodesic treewidth of a graph $ G $ is the smallest $k$ for which there is a partition $\mathcal{P}$ into geodesics such that $G/\mathcal{P}$ has treewidth $k$, where $G/\mathcal{P}$ is obtained from $ G $ by contracting each part of $ \mathcal{P} $.
    Based on this notion, row treewidth was developed and is defined for a graph $ G $ as the smallest $ k $ such that $ G \subseteq H \boxtimes P $ for some graph $ H $ of treewidth $ k $ and a path $ P $.
    Equivalently, the row treewidth of a graph $ G $ is the smallest $ k $ for which there is a partition $ \mathcal{P} $ into disjoint unions of geodesics that are aligned with respect to some layering such that $ G/\mathcal{P} $ has treewidth $ k $.

    We separate the two notions by showing that bounded row treewidth does not imply bounded geodesic treewidth and by presenting a polynomial-time algorithm to decide whether a graph of treewidth 2 has geodesic treewidth 1, which is known to be \NP-hard for row treewidth [Biedl, Eppstein, Ueckerdt, 2025].
    More generally, we provide an algorithm to decide whether a given graph has geodesic treewidth at most $ d $ that is \XP\ in the treewidth, whereas there is no such algorithm for row treewidth, unless $ \P = \NP $ [Biedl, Eppstein, Ueckerdt, 2025].
    On the other hand, we show that computing the geodesic treewidth is \NP-hard and that every graph with geodesic treewidth 1 has bounded row treewidth.
    Moreover, we improve the best known lower bound on the geodesic treewidth of planar graphs to 5.
\end{abstract}

\section{Introduction}
\label{sec:Introduction}

Bounded \emph{treewidth} allows using a wide range of algorithmic techniques to solve numerous important problems efficiently.
However, many common graph classes, like planar graphs, do not have bounded treewidth. Hence, we are interested in generalizations of treewidth that apply to broader classes of graphs but still provide comparable structural benefits.
In particular, we need to be able to handle large grids as they are planar graphs of high treewidth.
A popular way to define such generalizations is to take a partition \calP of a graph $ G $ into (disjoint unions of) geodesics\footnote{A \emph{geodesic} is a shortest path between any two vertices in the graph.} and aim for small treewidth of the \emph{quotient} $ G / \calP $, i.\,e., the graph that is obtained by contracting each (possibly disconnected) part of \calP into a single vertex.
For a first example, note that a grid admits a partition into geodesics such that the quotient is a path and thus has small treewidth.
The two generalizations of treewidth we consider in this paper are a concept introduced by Pilipczuk and Siebertz~\cite{PILIPCZUK2021111}, which we call geodesic treewidth, and row treewidth, introduced by Dujmović, Joret, Micek, Morin, Ueckerdt, and Wood~\cite{PlanarGraphsQueueNumber} and Bose, Dujmović, Javarsineh, Morin, and Wood~\cite{Bose_2022}.

The \emph{geodesic treewidth} of a graph $ G $ is the smallest $k$ for which there is a partition $\mathcal{P}$ into geodesics
such that $G/\mathcal{P}$ has treewidth $k$, see \cref{fig:GeodesicStructure}.
This concept is first introduced by Pilipczuk and Siebertz~\cite{PILIPCZUK2021111} to prove structural and algorithmic results on proper minor-closed graph classes, in particular planar graphs, concerning $ p $-centered colorings, bounded expansion, and subgraph-isomorphism.
If a graph class has bounded geodesic treewidth, we say it admits \emph{geodesic structure}.

\begin{figure}
    \centering
    \includegraphics{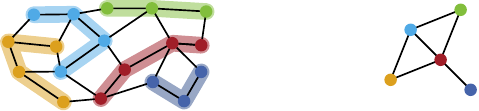}
    \caption{%
    Left: A graph $G$ with a partition $\mathcal{P}$ into geodesics.
    Right: The graph $G/\mathcal{P}$. 
    }       
    \label{fig:GeodesicStructure}
\end{figure}

Based on geodesic structure, Dujmović, Joret, Micek, Morin, Ueckerdt, and Wood~\cite{PlanarGraphsQueueNumber} develop the concept of product structure, which led to many structural~%
\cite{%
queue_number_42,
ProductStructurekPlanarGraphs,
hframedgraphs,
distel2023powers,
ShallowMinorsUSW,
hendrey2025structurekmatchingplanargraphs,
ShorterLabelingSchemesforPlanarGraphs,
AdjacencyLabellingforPlanarGraphsAndBeyond,
SparseUniversalGraphsForPlanarity,
Dujmovi__2020,
Dbski2021,
DEMW-23,
dujmović_esperet_morin_walczak_wood_2022,
VertexRankingPlanarGraphs,
ReducedBandwidth,
DGLTU-22,
JR-23,
KPS-24,
OddColouringsofGraphProducts,
goetze2025strongoddcoloringminorclosed,
hyperbolic,
square_graphs} 
and algorithmic~\cite{product_structure_algorithm,planar_linear,planar_nlogn,biedl2023complexityembeddinggraphproducts} results.
In contrast to geodesic structure, product structure does not require a single geodesic in each part, but instead disjoint unions of geodesics, where the parts are aligned by a layering. For this, a partition \calP of a graph has \emph{layered width} $ 1 $ with respect to some layering if each part of \calP contains at most one vertex in each layer.
The \emph{row treewidth} of a graph $ G $ is the smallest $k$ for which there is a partition $\mathcal{P}$ of layered width 1 such that $G/\mathcal{P}$ has treewidth $k$, see \cref{fig:ProductStructure}.
We say a graph class admits \emph{product structure} if it has bounded row treewidth.
Another perspective on product structure is that a graph $ G $ has row treewidth at most $ k $ if and only if $ G $ is a subgraph of the strong product $ H \boxtimes P $ of a graph $ H $ of treewidth $ k $ and a path $ P $.

\begin{figure}
    \centering
    \includegraphics{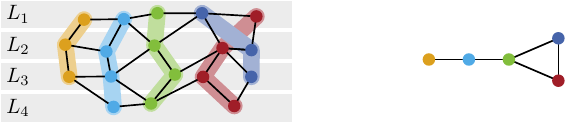}
    \caption{%
    Left: A graph $G$ with a layering $\mathcal{L}=(L_1, L_2, L_3, L_4)$ and a partition $\mathcal{P}$ of layered width~1. Right: The corresponding quotient $G/\mathcal{P}$ resulting from contracting each part of $\mathcal{P}$.
    }       
    \label{fig:ProductStructure}
\end{figure}

Although product structure has evolved from geodesic structure, only few results are known on the relationship between these two notions.
Most prominently, it is known that planar graphs and bounded-genus graphs admit both geodesic structure and product structure~\cite{PlanarGraphsQueueNumber,PILIPCZUK2021111}.
More general, for proper-minor closed graph classes, geodesic structure implies product structure since geodesic structure implies linear local treewidth%
\footnote{%
This follows from the definitions. See \cref{sec:Introduction:Preliminaries} for the definition and a brief argument.
}
(even for general graphs),
which is equivalent to product structure for proper minor-closed graph classes~\cite{PlanarGraphsQueueNumber}.
It is also remarkable that the best known upper bounds on the geodesic treewidth and row treewidth of planar graphs is 6 in both cases and is shown with a single construction that proves both upper bounds by providing a partition into geodesics of layered width 1~\cite{ImprovedPlanarGraphProductStructureTheorem}.
Moreover, such a partition can be computed in linear time~\cite{planar_linear}.
Similarly, the lower bound of 3 is certified by the same graph for both parameters~\cite{PlanarGraphsQueueNumber}.
While these results indicate that geodesic structure and product structure coincide on some interesting graph classes, in this paper we separate the two concepts both structurally and algorithmically.
In particular, we show that geodesic treewidth is interesting from an algorithmic perspective since it is less hard to compute than row treewidth.

We assume familiarity with basic graph concepts and refer to \cref{app:preliminaries} for details.

\subsection{Contributions}
\label{sec:Introduction:Contributions}

We compare product structure to its predecessor, geodesic structure, and find that the two are not equivalent.
For this, we show that a number of interesting graph classes admit product structure but not geodesic structure, including 1-planar graphs and graphs of row~treewidth~1.

\begin{restatable}{theorem}{psNotGs}\refcommand{sec:PSvsGS:pSnotGS}\label{thm:psNotGs}
    Product structure does not imply geodesic structure.
\end{restatable}
In addition, it is known that determining the row treewidth is \NP-hard even for treewidth-2 graphs~\cite{biedl2023complexityembeddinggraphproducts}, so there is no \XP-algorithm parametrized in the treewidth of the input graph, unless $ \P = \NP $.
In contrast, we give such an \XP-algorithm for determining the geodesic treewidth.
In particular, this yields a polynomial-time algorithm for the geodesic treewidth of treewidth-2 graphs.

\begin{restatable}{theorem}{OPTriandGS}\refcommand{sec:XPAlgorithm}\label{thm:OPTriandGS}
    There is an $O(n^{2k+3} f(k))$-time algorithm that decides for a graph $ G $ of treewidth $ k $ and $ d \in \mathbb{N}$ whether $G$ has geodesic treewidth at most $d$, where $ f $ is a function.
\end{restatable}

On the other hand, row treewidth and geodesic treewidth are closely related in other aspects. 
On the algorithmic side, we show that geodesic treewidth, too, is \NP-hard.

\begin{restatable}{theorem}{GSNPHigherKNewVersion}\refcommand{sec:GS:NpHard}\label{thm:GSNPHigherK}
    Determining if a graph has geodesic treewidth at most $k$ is \NP-hard for $k \geq 2$.
\end{restatable}

Complementing \cref{thm:psNotGs}, it is an interesting question whether geodesic structure implies product structure.
For the special case of graph classes with geodesic treewidth 1 we show that they also admit product structure, which is not the case in reverse, i.\,e., there are graph classes with row treewidth 1 but no geodesic structure.

\begin{restatable}{theorem}{GsImpliesPSForTrees}\refcommand{sec:PSvsGS:doesGsimplyPs}\label{thm:GsImpliesPSForTrees}
    Every graph with geodesic treewidth $1$ has row treewidth at most $7$.
\end{restatable}

Lastly, we consider the row treewidth and geodesic treewidth of planar graphs. The best previously known lower bound for the row treewidth and geodesic treewidth of planar graphs is 3 \cite{ImprovedPlanarGraphProductStructureTheorem}.
For geodesic structure, we show a lower bound of 5, almost matching the upper bound of 6~\cite{biedl2023complexityembeddinggraphproducts}.

\begin{restatable}{theorem}{LbPlanarGraph}\refcommand{sec:GS:GSLowerBoundPlanar}\label{thm:LbPlanarGraph}
    There is a planar graph whose geodesic treewidth is at least $5$.
\end{restatable}

\renewcommand{\refcommand}[1]{}

\section{Product Structure does not imply Geodesic Structure}
\label{sec:PSvsGS:pSnotGS}

We show that product structure does not imply geodesic structure by observing how geodesic structure and product structure behave differently when subdividing edges. 

Let $\mathcal{G}$ be a class of graphs. We define $\mathcal{G}' = \{G' \mid G \in \mathcal{G}\}$ where $G'$ is the result of taking graph $G$ and subdividing each edge $|V(G)|$ times. 
This reduces the row treewidth.
This is somewhat in contrast to a result of Bose, Dujmović, Javarsineh, Morin, and Wood~\cite{Bose_2022} who separate layered treewidth from row treewidth by showing that subdividing edges \enquote{efficiently} reduces the layered treewidth of a graph but not the row treewidth.

\begin{lemma} \label{lem:PSwithSubDiv}
For every graph class $\mathcal{G}$, the class $\mathcal{G}'$ has row treewidth $1$.
\end{lemma}

\begin{proof}
    We consider a complete graph $K_n$ and its subdivision $K_n'$, and show that $K_n' \subseteq S \boxtimes P$ where $P$ is a path and $S$ is a subdivided star. In particular, the graph $S$ is a tree and, thus, $K_n'$ has row treewidth $1$. Let $P$ be the path on $n$ vertices $p_1, \dots , p_n$ and $S$ the result of taking a star with center $c$ and $n(n-1)/2$ leafs and subdividing each edge to become an $ n $-vertex path. We embed the graph $K_n'$ in $S \boxtimes P$ as follows. For the vertices $v_i \in V(K_n)$, we embed them as $(c, p_i)$. We assign each edge $v_iv_j$ with $i < j$ in $K_n$ a leaf $l_{i,j}$ of $S$. Let $S_{i,j}$ be the path from $l_{i, j}$ to $c$ in $S$. We embed the subdivided edge between $v_i$ and $v_j$ in $K_n'$ in the subgraph $S_{i, j} \boxtimes P \subseteq S \boxtimes P$ without using vertices $(c, p_k)$ for $k \notin \{i, j\}$. 
    Note that $ S_{i, j} \boxtimes P $ is an $ n \times n $-grid with diagonals and indeed contains the required $ p_i $-$ p_j $-path with $ n $ internal vertices.
    Thus, it holds that $K_n' \subseteq S \boxtimes P$. Since this holds for complete graphs, it also holds for all other graphs and $\mathcal{G}'$ has row treewidth $1$ for every graph class $\mathcal{G}$. 
\end{proof}

For geodesic treewidth, it is not the case that subdivisions necessarily decrease it.

\begin{lemma}\label{lem:NoGS}
    For a graph class $\mathcal{G}$ that does not admit geodesic structure, the class $\mathcal{G}'$ also does not admit geodesic structure.
\end{lemma}

\begin{proof}
    Assume $\mathcal{G}'$ does admit geodesic structure with geodesic treewidth $c$. That is, for each $G' \in \mathcal{G'}$ there exists a partition $\mathcal{P'}$ of $G'$ into geodesics such that $G'/\mathcal{P'}$ has treewidth at most $c$. We consider a geodesic $P'\in \mathcal{P'}$ and define $P=P'\cap V(G)$. We show that $P$ is a geodesic in $G$. If $P$ is not a geodesic in $G$, then there exists a shorter path $S$ from the start of $P$ to the end of $P$. Let $S'$ be the path corresponding to the path $S$ in $G'$. However, since $G'$ is the result of subdividing each edge $|V(G)|$ times, it holds that $S'$ is a shortcut for $P'$. Since this is impossible, it holds that $P$ is a geodesic in $G$. Therefore, $\mathcal{P}=\{P'\cap V(G) \mid P'\in \mathcal{P'}\}$ is a partition of $G$ into geodesics. It holds that $G/\mathcal{P}$ is a minor of $G'/\mathcal{P'}$ and thus has treewidth at most $\tw(G'/\mathcal{P'})=c$. Thus, the geodesic treewidth of $G$ is at most $c$. Therefore, $\mathcal{G'}$ admitting geodesic structure implies that $\mathcal{G}$ admits geodesic structure.
\end{proof}

Note that with \cref{lem:PSwithSubDiv,lem:NoGS} we obtain a rich collection of examples that admit product structure with row treewidth 1 but not geodesic structure by taking any graph class without geodesic structure and subdividing sufficiently often, for instance complete graphs.

\begin{corollary}\label{cor:rtw1gtwInf}
    There are graph classes with row treewidth $1$ and unbounded geodesic treewidth.
\end{corollary}

As an additional corollary we observe that for every graph class $ \mathcal{G} $, the class $\mathcal{G}'' = (\mathcal{G}')'$ is 1-planar since each edge of a graph $G \in \mathcal{G}$ is subdivided at least $|V(G)|^2$ times and thus can take $|V(G)|^2$ crossings. Hence, the 1-planar graphs include $ \mathcal{G}''$ for $ \mathcal{G} $ being the complete graphs, which do not admit geodesic structure. 
By \cref{lem:NoGS}, the latter transfers to~$ \mathcal{G}''$.
\begin{corollary}\label{cor:1planarGtwInf}
    The class of $1$-planar graphs does not admit geodesic structure.
\end{corollary}
Since 1-planar graphs admit product structure~\cite[Theorem 3]{ProductStructurekPlanarGraphs}, 
they provide another example for separating geodesic structure from product structure.
We remark that \cref{cor:1planarGtwInf} implies that many other beyond-planar graph classes that admit product structure, like $k$-planar graphs~\cite{ProductStructurekPlanarGraphs}, fan-planar graphs, fan-bundle planar graphs~\cite{ShallowMinorsUSW}, and $k$-matching planar graphs~\cite{hendrey2025structurekmatchingplanargraphs} also do not admit geodesic structure. \Cref{cor:rtw1gtwInf} as well as the mentioned beyond-planar graph classes certify \cref{thm:psNotGs}.

\psNotGs*

\section{XP-Algorithm for Computing the Geodesic Treewidth}
\label{sec:XPAlgorithm}

\newcommand{\tominor}{\ensuremath{\mathtt{AssignToMinor}}\xspace}
\newcommand{\consistent}[1]{\mathtt{Consistency}(#1)}
\newcommand{\components}{\ensuremath{\mathtt{Components}}\xspace}
\newcommand{\obssub}{\mathcal{S}_d}
\newcommand{\minorvertex}[2]{V_{#1}(#2)}

Biedl, Eppstein, and Ueckerdt~\cite{biedl2023complexityembeddinggraphproducts} show that computing the row treewidth is \NP-hard even for treewidth-$2$ graphs.
In contrast, we show that the geodesic treewidth of treewidth-$2$ graphs can be computed in polynomial time.
Moreover, we show that for all $k \in \mathbb{N}$ there is a polynomial-time algorithm that computes the geodesic treewidth of graphs with treewidth~$k$.

\OPTriandGS*

To show this theorem we use that treewidth-$d$ graphs can be characterized via forbidden minors, i.\,e., there is a set $ \obs $ such that a graph has treewidth at most $d$ if and only if it contains no graph in $\obs$ as a minor, and $ \obs $ can be computed in time dependent only on $d$~\cite{computing_obstructions2,computing_obstructions}.
While the best known upper bound on the size of $ \obs $ is doubly exponential in $O(d^5)$~\cite{computing_obstructions}, these lists are quite small for small $ d $.
For example, for $d=1$ we have $\mathcal{G}_1=\{K_3\}$ as a graph is a forest if and only if it does not contain $K_3$ as a minor, $\mathcal{G}_2=\{K_4\}$~\cite{BODLAENDER19981}, and $\mathcal{G}_3 $ contains $ K_5 $, the octahedron, the pentagonal prism graph, and the Wagner graph~\cite{forbidden_minors_3-trees,characterization_3-trees}.

To check whether $G$ admits a partition \calP into geodesics such that the quotient $ G / \calP $ does not contain a minor from $ \obs $, we construct a dynamic program on a nice tree decomposition%
\footnote{A \emph{nice tree decompositions} has leaf bags, introduce bags, forget bags, and join bags, see \cref{sec:Introduction:Preliminaries}.}
of $G$.
That is, we process the tree decomposition from bottom to top.
Before we describe our algorithm, let us introduce some notation.
We remark that when using the following definitions, we choose $ J \subseteq G / \calP $, where $ G $ is the input graph, \calP a partition of $ G $ into geodesics, and $ \mathcal{Q} $ is a partition of the quotient $ J $ which we aim to use to identify a minor of $ J $.
During the algorithm it will be convenient to think of vertices of $J$ as the subsets of $V(G)$ that are contracted to form the quotient. 
Thus, for this section we assume that for such a quotient $J \subseteq G / \calP $ we have $V(J) \subset 2^{V(G)}$.
Note that if each part of $ \mathcal{Q} $ induces a connected subgraph of $ J $, then $ J/\mathcal{Q} $ is a minor of $ J $.
However, in the course of the algorithm, the parts in $ \mathcal{Q} $ are not necessarily connected, which we handle with the following definitions.
For a quotient $J/\mathcal{Q}$ with $\mathcal{Q}$ being a partition of $J$, we refer to the part $Q\in\mathcal{Q}$ that is contracted to form a vertex $v \in V(J/\mathcal{Q})$ as the \emph{improper branch set} of $v$.
Note that improper branch sets do not have to be connected, and thus $J/\mathcal{Q}$ is not necessarily a minor of $J$, hence the name \emph{improper} branch set.
Next, let $(H,\components)$ be a tuple with $H$ being a graph and $\components$ being a function $V(H) \rightarrow 2^{2^{V(J)}}$, i.\,e., $\components$ assigns each vertex of $H$ a set of subsets of $V(J)$.
We call $(H,\components)$ an \emph{improper minor} of some graph $J$ if there is a partition $\mathcal{Q}$ of $J$ such that $J/\mathcal{Q}$ is isomorphic to $H$ and for each vertex $v \in V(J/\mathcal{Q})$ the improper branch set of $v$ induces a subgraph of $J$ whose connected components are precisely $\components(v)$.
That is, $H$ is a minor of $J$ if and only if there is an improper minor $(H,\components)$ of $J$ such that $|\components(h)|=1$ for all $h \in V(H)$.
In our algorithm we use improper minors to keep track of possible minors of a quotient of our input graph.
However, we remark that while it is important for us to maintain some improper minors with disconnected improper branch sets, at the end of the algorithm we check whether the improper minors are minors, so they should be thought of as potential minors.

During the algorithm, for each bag $B$ of the tree decomposition of $G$ and the subtree $T_B$ below $B$ (including $B$) we store the following:
First, we store some information about all possible partitions of $G[T_B]$ into geodesics, where $ G[T_B] $ is the graph induced by the bags of $ T_B $.
Clearly, we cannot store all of these partitions as the number of these partitions is exponential in $n$.
Instead, for each partition $\mathcal{P}_T$ of $G[T_B]$ into geodesics, we store a partition $\mathcal{P}_B$ of $B$ with some additional information.
For each geodesic $P_T\in \mathcal{P}_T$ that is not disjoint from $B$ we store a part $P_B$ in $\mathcal{P}_B$ that contains the vertices $V(P_T)\cap B$ that are ordered according to their order in $P_T$.
We refer to the first and last of these vertices as the \emph{border-vertices} of $P_B$.
Then, for each part $P_B$ with corresponding geodesic $ P_T $ in $ G[T_B] $ we store in a label attached to $P_B$ whether both end-vertices of $P_T$ are not in $B$, both end-vertices of $P_B$ are in $B$ or exactly one end-vertex $ v $ of $P_T$ is not in $B$. In the last case we also store $ v $ and the border-vertex of $P_B$ it is closest to.
\Cref{fig:DP_partitions} shows an example.

\begin{figure}
        \centering
        \begin{subfigure}[t]{0.33\textwidth}
            \centering
            \includegraphics[page=1]{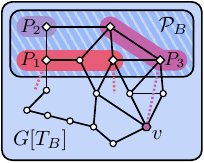}
            \subcaption{}
            \label{fig:DP_bag_partition}
        \end{subfigure}%
        \hfill
        \begin{subfigure}[t]{0.33\textwidth}
            \centering
            \includegraphics[page=2]{figures/DP_configurations_new.pdf}
            \subcaption{}
            \label{fig:DP_partition1}
        \end{subfigure}%
        \hfill
        \begin{subfigure}[t]{0.33\textwidth}
            \centering
            \includegraphics[page=3]{figures/DP_configurations_new.pdf}
            \subcaption{}
            \label{fig:DP_partition2}
        \end{subfigure}%
        \caption{
            \subref{fig:DP_bag_partition} An example for a bag $B$ (hatched blue), a graph $G[T_B]$, and a partition $\mathcal{P}_B$ of $B$ with the border-vertices of each part depicted as diamonds.
            The part $P_1$ has two end-vertices outside of $B$ (that we do not store) indicated by the dashed lines.
            The part $P_2$ has no end-vertices outside of $B$.
            The part $P_3$ has one end-vertex $v$ outside of $B$ depicted as an $8$-gon in the same color as $P_3$. It is associated with one of the border-vertices of $P_3$ indicated by the dashed line.
            \textbf{\sffamily(b,c)}
            Two different partitions $P_{T_1}$ and $P_{T_2}$ of $G[T_B]$ into geodesics that results in $\mathcal{P}_B$ if restricted to $B$.}
        \label{fig:DP_partitions}
    \end{figure}

Note that this gives us at most $2n$ possible labels
for each part $P_B$ as there are $ n - 1 $ options for the end-vertex, combined with two options for the border-vertex, and additionally the two options that none or both end-vertices are in $ B $.
Since each partition of $B$ has at most $ |B| \leq k+1$ parts, there are at most $(2n)^{k+1}$ different combinations of labels for a fixed partition of $B$.
While we need to store more things, everything in the following only depends on $k$.
Thus, since the number of partitions of $B$ also depends only on $k$, in total we obtain $O(n^{k+1}\cdot f(k))$ distinct configurations for some function $f$.

In addition to the information about a partition $\mathcal{P}_T$ of $G[T_B]$ we already have, we store information about improper minors of $G[T_B]/\mathcal{P}_T$.
To do this, let us define $\obssub$ as the set of all labeled subgraphs of graphs in $\obs$.
That is, we distinguish two isomorphic subgraphs if they are on different vertex-sets, e.\,g., a graph $H\in \obs$ has $|V(H)|$ distinct subgraphs that consist of a single vertex.
Recall that an improper minor of $J = G[T_B]/\mathcal{P}_T$ is a tuple $(H,\components)$ with $H$ being a graph and $\components$ being a function $V(H)\rightarrow2^{2^{V(J)}}$.
Our goal is to store all improper $\obssub$-minors of $G[T_B]/\mathcal{P}_T$, i.\,e., all improper minors $(H,\components)$ of $G[T_B]/\mathcal{P}_T$ with $H\in \obssub$.
However, since we want the number of these to depend only on $k$, we instead store something slightly different.
That is, for each improper $\obssub$-minor $(H,\components)$, we store the \emph{$B$-improper minor} which is the tuple $(H,\components')$ with $\components'(v)=\{ \{P \cap B \colon P \in C\} \colon C \in \components(v)\}$ for each vertex $v\in V(H)$, where $\components'(v)$ is a multi-set and $ | \components'(v) | = | \components(v) | $.
This means we only store the intersection of each geodesic in $\mathcal{P}_T$ with $B$ instead of storing the entire geodesic. 
Moreover, if there is a vertex $v\in V(H)$ such that $\{\emptyset\} \in \components'(v)$ and $|\components'(v)|> 1$ we say that $(H,\components')$ is an \emph{incompletable} $B$-improper minor and discard the tuple instead (cf.\ \cref{fig:DP_improper_minor}).
Otherwise, we refer to $(H,\components')$ as a \emph{completable} $B$-improper minor of $G[T_B]/\mathcal{P}_T$.
Note that for a completable $B$-improper minor $(H,\components')$, the set $\components'(v)$ contains each set only once for every $v \in V(H)$.
Indeed, the multi-set is only needed to properly recognize the case that $\components'(v)$ contains $\{\emptyset\}$ more than once for some vertex $v \in V(H)$.
This does not occur during the algorithm and, thus, we ignore it in the following and always assume that $\components'$ maps to a set.
Note that the number of completable $B$-improper minors depends only on $k$.

\begin{figure}
    \centering
    \includegraphics[page=1]{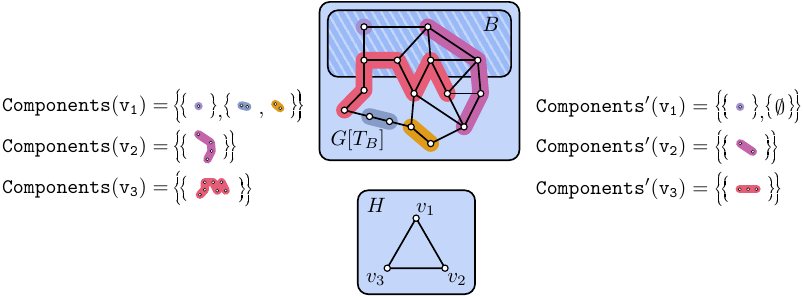}
    \caption{Middle: A graph $G[T_B]$ with a partition into geodesics and a graph $H$.
    Left: An improper minor $(H,\components)$.
    Right: The corresponding (incompletable) $B$-improper minor $(H,\components')$.
    Each colored element in $\components$ and $\components'$ is a set of vertices.}
    \label{fig:DP_improper_minor}
\end{figure}

In total for each partition $\mathcal{P}_T$ of $G[T_B]$ into geodesics with the set of improper $\obssub$-minors of $G[T_B]/\mathcal{P}_T$ we define a \emph{configuration} of the bag $B$ that consists of the following things:

\begin{itemize}
    \item A partition $\mathcal{P}_B$ of $B$ (corresponding to $\mathcal{P}_T$) with each part $P\in\mathcal{P}_B$ being ordered,
    \item a set $L$ of labels indexed by the parts $P\in\mathcal{P}_B$ that tells us that two or none end-vertices associated with $P$ are outside of $B$, or that associates a border-vertex of $P$ with a vertex outside of $B$, and
    \item a set $M$ of completable $B$-improper $\obssub$-minors (corresponding to the improper $\obssub$-minors of $G[T_B]/\mathcal{P}_T$). 
\end{itemize}

Taking the union over all partitions $ \calP_T $ of $ G[T_B] $ into geodesics, we obtain the set of all configurations of $ B $.
Before computing the configurations, we argue why it suffices to do so.

\begin{lemma}\label{lem:RootCorrect}
    Let $G$ be a graph with a nice tree decomposition with root bag $B$.
    Then, the graph $G$ has geodesic treewidth at most $d$ if and only if $B$ has a configuration with a set $M$ of completable $B$-improper $\obssub$-minors such that for each $(H,\components')\in M$ we have $H \notin \obs$ or $|\components'(v)|>1$ for a vertex $v\in V(H)$.
\end{lemma}

\begin{proof}
    Let $G$ have geodesic treewidth at most $d$.
    Then, by definition, there is a partition $\mathcal{P}$ of $G$ into geodesics such that $G/\mathcal{P}$ has treewidth at most $d$.
    As $\obs$ contains precisely the forbidden minors of treewidth $d$, the graph $G/\mathcal{P}$ contains no graph $H\in\obs$ as a minor.
    Thus, by the definition of improper minors, we have for every improper minor $(H',\components)$ of $G/\mathcal{P}$ that $H'\notin\obs$ or $|\components(v)|>1$ for a vertex $v\in V(H')$. 
    In the first case we are done, and in the second we have $ |\components'(v)| = |\components(v)| > 1 $.

    Now, let $G$ have geodesic treewidth at least $d+1$.
    Then, by the definition of geodesic treewidth, for every partition $\mathcal{P}$ of $G$ into geodesics the graph $G/\mathcal{P}$ has treewidth at least $d+1$.
    Thus, the graph $G/\mathcal{P}$ contains some graph from $\obs$ as a minor.
    Now, consider a configuration of $B$ with partition $\mathcal{P}_B$ of $B$ and a set $M$ of completable $B$-improper $\obssub$-minors.
    Our goal is to show that there is a completable $B$-improper minor $(H,\components')\in M$ with $H \in \obs$ and $|\components'(v)|=1$ for each $v\in V(H)$.
    To do this recall that by definition, a configuration results from a partition \calP of $ G[T_B] = G $.
    By assumption, we have $ \tw(G/\calP) \geq d + 1 $ and thus $ G/\calP$ has a minor $H \in \obs$.
    Now let $(H,\components)$ be the improper minor with $|\components(v)|=1$ for each $v\in V(H)$ obtained from the branch sets of $ H $, which exists by definition.
    Consider the corresponding $B$-improper minor $(H,\components')$.
    Since $|\components'(v)| = |\components(v)|=1$ for each $v \in V(H)$, we know that $(H,\components')$ is completable and, thus, that $(H,\components')\in M$. 
\end{proof}

Let us now come to a summary of how our algorithm computes the set of configuration for all bags of a nice tree decomposition of width $ k $.
A detailed description and a proof of correctness is in \cref{app:correctness}.
We start with a pre-computation of the distances between any two vertices, verify $ d \leq k $ (we are done otherwise), and compute the sets $ \obs $ and $ \obssub $.

Leaf bags are brute-forced.
For an introduce bag $B$ with child $B_1=B - \{v\}$ and a configuration with a partition $ \calP_1 $ of $B_1$, we consider every way to add $ v $. 
We can add $\{v\}$ as a new part, add it to an existing part (if it remains a geodesic) or merge two parts with $v$ in between (if this results in a geodesic). 
For a resulting partition of $B$ we construct a configuration of $B$ by iterating over all completable $B_1$-improper minors in the configuration of $B_1$ and consider every way to add $ v $ to them. 
Similarly, we handle forget bags by checking whether removing the vertex yields a completable $B$-improper minor.
For a join bag $B$ with children $B_1$ and $B_2$ we iterate over all pairs of configurations of $B_1$ and $B_2$ with partitions $ \calP_1 $, $ \calP_2 $.
If $ \calP_1 = \calP_2 $ and the labels agree, we check for all pairs of completable $B_1$- and $ B_2 $-improper minors whether we can combine them to a completable $B$-improper minor.

We remark that the join bags are crucial for the runtime, taking $ O(n^{2k+2} \cdot f(k) ) $ time for all pairs of $ O(n^{k+1} \cdot f(k) ) $ configurations.
With $ O(n) $ bags, we obtain a runtime of $O(n^{2k+3}\cdot f(k))$.
We discuss in \cref{rem:FPT} that the geodesics are a bottleneck as we obtain an $O(n \cdot f(k))$ algorithm if we want to partition $G$ into any induced paths instead of geodesics.

\section{Computing the Geodesic Treewidth is NP-Hard}
\label{sec:GS:NpHard}

We complement our algorithm by showing that, similar to row treewidth, it is \NP-hard to decide whether a graph has geodesic treewidth $2$.
For this, we give a reduction from a SAT variant that is known to be \NP-complete. Tovey~\cite{TOVEY198485} shows that SAT is \NP-complete if every clause contains $2$ or $3$ variables and every variable occurs at most once in a clause of size $3$ and at most twice in clauses of size $2$. We consider such an instance $I$ of SAT with clauses $C$ and variables $U$ and construct the following graph. We begin with a path $P$ of length $2n+1$ where $n=|U|$. Let the first vertex on this path be called $s$ and the last $t$. Beginning with the vertex adjacent to $s$ we now duplicate every other vertex on the path, meaning we add new vertices with the same neighborhood as the vertices we are duplicating. Let the resulting graph be called $G_1$ and let $\{(x_i, \overline{x_i}) \mid x_i \in P \wedge \overline{x_i} \text{ is the duplicate vertex of } x_i\}$ be the pairs of vertices that are duplicated. An example of such a graph $G_1$ is shown in \Cref{fig:NPG1} on the left. For an $s$-$t$-path $S$, we say $x_i \in U$ is true if and only if $x_i \in S$.
Thus, every path $S$ corresponds uniquely to a truth assignment of the variables $U$. 

For each clause $c \in C$ and literals $l_1, l_2 \in c$, we add the edge $l_1l_2$ to $G_1$ to construct $G_2$ as shown in \Cref{fig:NPG1} in the center.
That is, each clause in $C$ corresponds to a clique in $G_2$.
In addition, for each clause $\{l_1, l_2\} \in C$ of size $2$ we add five new vertices $p_1, \dots, p_5$ that are adjacent to $l_1, l_2,$ and $s$. The resulting graph $G_3$ can be seen in \Cref{fig:NPG1} on the right. 
Lastly, to construct $G_4$, every edge $e \in E(G_3)\setminus E(G_1)$ is subdivided $2n$ times.
As a result, every shortest path from $s$ to $t$ in $G_4$ only uses edges of $G_1$ and, thus, we retain the unique correspondence between shortest paths from $s$ to $t$ and truth assignments of the variables.

\begin{figure}[htb]
    \centering
    \includegraphics{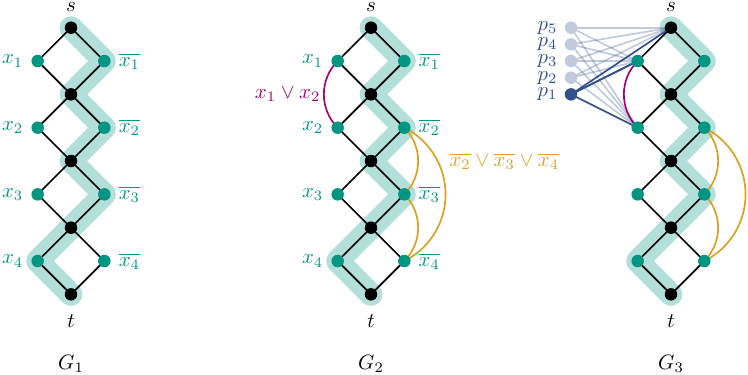}
    \caption{%
    Left: An example of $G_1$ with a shortest $s$-$t$-path highlighted in green that uniquely corresponds to the truth assignment $\varphi$ with $\varphi(x_4)=t$ and $\varphi(x_i)=f$ for $i \in [3]$.
    Middle: A corresponding $G_2$ with a SAT instance consisting of two clauses.
    Right: A corresponding $G_3$.
    }       
    \label{fig:NPG1}
\end{figure}

For a graph $G$ with $s,t \in V(G)$, let an \emph{$s$-$t$-partition} of $G$ into geodesics be a partition into geodesics that contains a shortest $s$-$t$-path. We show the following in \Cref{app:GS:NpHard}.

\begin{restatable}{lemma}{SPGSNP}\label{lem:SP-GSNP}
    The instance $I$ is satisfiable if and only if there exists an $s$-$t$-partition $\mathcal{P}$ of $G_4$ into geodesics such that $G_4/\mathcal{P}$ has treewidth at most $2$.
\end{restatable}

\begin{proof}[Proof sketch]
    If $I$ is not satisfiable then for every $s$-$t$-partition with shortest $s$-$t$-path $S$ there is a clause $c \in C$ such that no literal of $c$ is in $S$. It can be seen that the geodesics containing the literals of $c$, the geodesics containing the corresponding vertices $p_1,\dots,p_5$ (if $c$ is a clause of size $2$), and the $s$-$t$-geodesic form a $K_4$ minor.

    If $I$ is satisfiable then deleting the vertices of an $s$-$t$-path $S$ corresponding to a satisfying truth assignment results in the disjoint union of some trees and paths.
    Thus, for any partition $\mathcal{P}$ of $G_4$ into geodesics that contains $S$, the graph $G_4/\mathcal{P}$ has treewidth at most $2$.
\end{proof}

Finally, we construct $G$ by subdividing a $K_4$ with vertices $ v_1, \dots, v_4 $ twice and replacing the six edges whose endpoints are subdivision vertices by copies of $ G_4 $ as shown in \cref{fig:NPG}.

\begin{figure}
    \centering
    \includegraphics{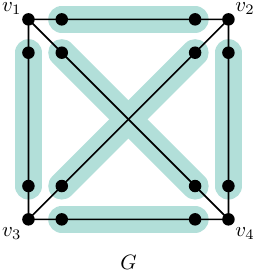}
    \caption{$K_4$ with vertices $v_1, \dots v_4$ where each highlighted edge gets replaced by a copy of $G_4$.}
    \label{fig:NPG}
\end{figure}

\begin{restatable}{lemma}{GSNP}\label{lem:GSNP}
    The graph $G$ has geodesic treewidth at most $2$ if and only if $I$ is satisfiable.
\end{restatable}

\begin{proof}[Proof sketch]
    If $I$ is satisfiable we take an $s$-$t$ partition $\mathcal{P}$ of $G_4$ such that $G_4/\mathcal{P}$ has treewidth at most $2$ (\cref{lem:SP-GSNP}) for each copy of $G_4$ in $G$.
    Then, we include $v_1,\dots,v_4$ in some of the $s$-$t$-paths as depicted in \cref{s:fig:geodesic_extension}.
    This yields a partition $\mathcal{P}'$ into geodesics such that $G/\mathcal{P}'$ (\cref{s:fig:hardness_quotient}) is a $K_{2,4}$ with some attached copies of $G_4/\mathcal{P}$ and, thus, has treewidth at most $2$.

    If $I$ is not satisfiable, by \cref{lem:SP-GSNP} no $s$-$t$ partition $\mathcal{P}$ of $G_4$ yields $ \tw(G_4/\mathcal{P}) \leq 2 $.
    Then, $v_1,v_2,v_3 $, and $v_4$ are in distinct geodesics, giving a $ K_4 $-minor which has treewidth~3.
    \begin{figure}
        \centering
        \begin{subfigure}[m]{0.4\textwidth}
            \centering
            \includegraphics[page=1]{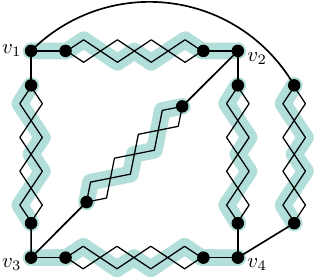}
            \subcaption{}
            \label{s:fig:geodesic_extension}
        \end{subfigure}%
        \hfill
        \begin{subfigure}[m]{0.56\textwidth}
            \centering
            \includegraphics[page=2]{figures/hardness_quotient.pdf}
            \subcaption{}
            \label{s:fig:hardness_quotient}
        \end{subfigure}%
        \caption{%
        \subref{s:fig:geodesic_extension} A sketch of $G$ with possibly extended $s$-$t$-paths highlighted in green, \subref{s:fig:hardness_quotient} the corresponding graph $G/\mathcal{P'}$. 
        The diamond vertices correspond to the green geodesics in \subref{s:fig:geodesic_extension}.
        }
        \label{s:fig:NPTD}
    \end{figure}
\end{proof}

\begin{corollary} \label{thm:GSNP}
    Determining whether a graph has geodesic treewidth at most $2$ is \NP-hard.
\end{corollary}

Next we generalize the \NP-hardness result from 2 to $k$ and show that determining whether a given graph has geodesic treewidth at most $k$ is \NP-hard. For a graph $G$, we define the graph $G^{\star}$ as follows: connect three disjoint copies $G_1, G_2, G_3$ of $G$ by a universal vertex $v$ and for each edge adjacent to $v$ subdivide it $|V(G)|$ times. 

\begin{restatable}{lemma}{GSInc}\label{lem:GSInc}
     $G$ has geodesic treewidth $k$ if and only if $G^{\star}$ has geodesic treewidth $k+1$.
\end{restatable}

\begin{proof}[Proof sketch]
    Due to the subdivisions, geodesics in $ G $ are also geodesics in $ G^\star $, so we may reuse the partition of $ G $ into geodesics and put $ v $ and the subdivision vertices into its own geodesic each.
    For the other direction, we consider a partition $ \calP^\star $ of $ G^\star $ into geodesics and find a suitable partition for $ G $ by restricting $ \calP^\star $ to one of the copies.
    In \cref{app:GS:NpHard} we show that the described partitions indeed yield quotients of treewidth $ k + 1 $, respectively $ k $.
\end{proof}

Now, the following theorem follows immediately from \Cref{thm:GSNP} and \Cref{lem:GSInc}.

\GSNPHigherKNewVersion*

\section{Geodesic Treewidth 1 implies Product Structure}

\label{sec:PSvsGS:doesGsimplyPs}

In this section, we show that graph classes with geodesic treewidth 1 have bounded row treewidth.
However, whether geodesic structure implies product structure in general remains an open question.
Note that this is in contrast to \cref{cor:rtw1gtwInf}, which shows that there are graphs with row treewidth 1 but unbounded geodesic treewidth.

For this, we define a \emph{$c$-layering} $\mathcal{L}$ of a graph $G$ as an ordered partition $(L_0, L_1, \dots)$ of $V(G)$ into \emph{layers} such that for every edge $vw \in E(G)$ with $v \in L_i$ and $w \in L_j$ we have $| i - j | \leq c$. 
We say a partition $ \mathcal{P} $ of some graph $ G $ has \emph{$c$-layered width} $d$ if there is a $c$-layering \calL such that each part of $ \mathcal{P} $ has at most $d$ vertices in each layer of $ \mathcal{L} $.
Note that a 1-layering is a layering and analogously 1-layered width coincides with layered width. 
Similarly to layered width~\cite{PlanarGraphsQueueNumber}, the notion of $c$-layered width turns out useful to upper-bound the row treewidth and thereby to prove \cref{thm:GsImpliesPSForTrees}.

\begin{restatable}{lemma}{GsImpliesPSForTreesLemma}\label{lem:GsImpliesPSForTrees}
    Let $G$ be a graph with a partition $\mathcal{P}$ of $G$ into geodesics such that $G/\mathcal{P}$ is a tree. 
    Then, \calP has $4$-layered width $1$.
\end{restatable}

\begin{proof}[Proof sketch]
    For a partition \calP of some graph into geodesics, we call a $c$-layering \calL \emph{aligned} if for each geodesic $P\in \mathcal{P}$ the vertices of $P$ are embedded in consecutive layers of $\mathcal{L}$ according to their order in $P$.
    Note that only edges in geodesics are required to have their endpoints in two consecutive layers, whereas edges between distinct geodesics are allowed skip up to $ c - 1 $ layers.
    We prove the following stronger statement by induction on the number of geodesics in \calP.
    Note that since by definition, every aligned $c$-layering has only one vertex of each geodesic in each layer, \cref{s:cl:alignedLayering} implies that the 4-layered width of \calP is 1.
    \begin{claim}\label{s:cl:alignedLayering}
        For every graph $ G $ and every partition \calP of $ G $ into geodesics such that $ G / \calP $ is a tree, there is an aligned 4-layering. 
    \end{claim}
    
    If $|\mathcal{P}| = 1$, then there is an aligned $4$-layering of $G$ by putting each vertex into its own layer, ordered according to the single geodesic in \calP. 
    Next, consider a graph $ G $ with a partition \calP into at least two geodesics such that $ G / \calP $ is a tree.
    \NewDocumentCommand{\leaf}{}{\ensuremath{\gamma}}%
    \NewDocumentCommand{\parent}{}{\ensuremath{\delta}}%
    Let $\leaf, \parent \in V(G/\mathcal{P})$ with $\leaf$ being a leaf adjacent to $\parent$. Let $P_\leaf$ and $P_\parent$ be the geodesics corresponding to $\leaf$ and $\parent$, respectively, with $P_\leaf=(v_1, \dots, v_m)$ and $P_\parent=(w_1, \dots, w_n)$.
    Let $\mathcal{L}=(L_0, L_1, \dots)$ be an aligned $4$-layering of $G-P_\leaf$ which we obtain by induction. 
    We construct two aligned $c$-layerings $\mathcal{L}_1, \mathcal{L}_2$ of $G$ from $\mathcal{L}$ and show that one of them indeed is an aligned 4-layering.
    For this, note that since $\leaf$ is adjacent to $\parent$ in $ G / \calP $, there exist $i \in [m]$ and $j \in [n]$ such that $e=v_iw_j \in E(G)$.
    We denote the layer in $\mathcal{L}$ that contains $w_j$ by $ L_x $ and use the edge $e$ to align the path $ P_\leaf $.
    That is, we put $ v_i $ in layer $ L_x $ and obtain two options how to continue $ P_\leaf $ from there.
    The two options are shown in \cref{s:fig:AlignTwoPaths} and yield the layerings $\mathcal{L}_1 $ and $ \mathcal{L}_2$.
    For a layering $\mathcal{L'}$ let a \emph{$k$-steep} edge be an edge with endpoints in layers $L_a, L_b$ of $\mathcal{L'}$ and $|a - b| = k$.

    \begin{figure}
        \centering
        \includegraphics{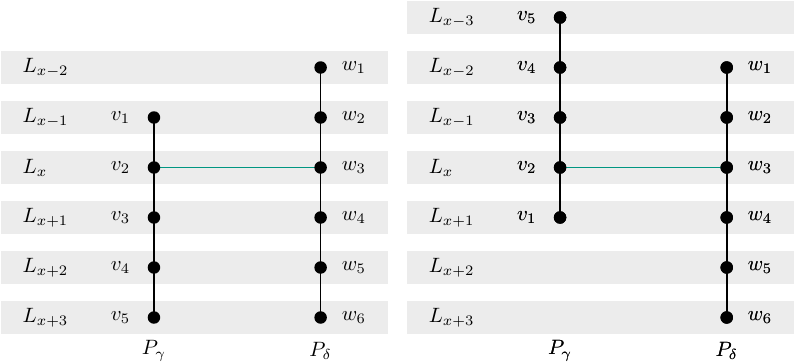}
        \caption{Layering $\mathcal{L}_1$ on the left and layering $\mathcal{L}_2$ on the right. The green edge is the alignment edge $e=v_iw_j$, where $ i = 2, j = 3 $ in this example.}
        \label{s:fig:AlignTwoPaths}
    \end{figure}

    We show that $\mathcal{L}_1 $ or $ \mathcal{L}_2$ is an aligned 4-layering, i.e., all edges are at most 4-steep.
    The key arguments are first that edges which do not cross $ e $ in the embedding of \cref{s:fig:AlignTwoPaths} are at most 2-steep (\cref{s:fig:NoSteepNonCrossingEdges}).
    And second, if $ \calL_1 $ contains a 5-steep edge $ e_1 $, then every edge $ e_2 $ crossing $ e $ in $ \calL_2 $ has its endpoints closer to $ e $ than $ e_1 $ (\cref{s:fig:NoNewSteepCrossingEdges}).
    Combining these arguments, we show in \cref{app:PSvsGS:doesGsimplyPs} that $ e_2 $ is at most 4-steep in $ \calL_2 $ by otherwise finding a shortcut for $ P_\parent $ using $ e_1 $ and $ e_2 $.
    \begin{figure}
        \centering
        \begin{subfigure}{0.21\textwidth}
            \includegraphics[page=2]{figures/alignedLayerings}
            \subcaption{}
            \label{s:fig:NoSteepNonCrossingEdges}
        \end{subfigure}
        \hspace{3em}
        \begin{subfigure}{0.21\textwidth}
            \includegraphics[page=1]{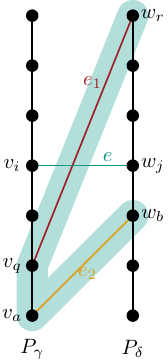}
            \subcaption{}
            \label{s:fig:NoNewSteepCrossingEdges}
        \end{subfigure}
        \caption{%
            Key arguments for \cref{lem:GsImpliesPSForTrees}.
            Recall that geodesics, in particular $ P_\leaf $ and $ P_\parent $, do not have shortcuts.
            \textbf{\textsf{(a)}} If a non-crossing edge $e'$ is $3$-steep, then together with edge $e$, this results in a shortcut, highlighted green.
            \textbf{\textsf{(b)}} Embeddings according to $ \calL_1 $ and assuming $a \geq q$. 
            Recall that $ e_2 $ crosses $ e $ in $ \calL_2 $ and thus does not cross $ e $ in $ \calL_1 $.
            Since $e_2$ is at most $2$-steep and $e_1$ is at least $5$-steep, the edges $e_1$ and $e_2$ result in a shortcut for $P_\parent$ that is highlighted in green.
        }
    \end{figure}
\end{proof}

It is now left to show that \cref{lem:GsImpliesPSForTrees} implies \cref{thm:GsImpliesPSForTrees}.

\GsImpliesPSForTrees*

\begin{proof}
    Let $G$ be a graph with a partition $\mathcal{P}$ of $G$ into geodesics such that $T = G/\mathcal{P}$ is a tree. 
    By \Cref{lem:GsImpliesPSForTrees}, there is a $4$-layering $\mathcal{L}$ of $G$ such that $\mathcal{P}$ has layered width $1$ with respect to $\mathcal{L}$. We partition the layers of $\mathcal{L}$ into consecutive blocks of size 4, where block $k$ contains layers $4k$ to $4(k+1)-1$, for $k \in \mathbb{Z}$, and merge the layers in a block to get a layering $\mathcal{L}'$. Note that $\mathcal{P}$ has layered width $4$ with respect to $\mathcal{L}'$. Thus $G \subseteq T \boxtimes P \boxtimes K_4$~\cite[Obs.~35]{PlanarGraphsQueueNumber} and therefore $G$ has row treewidth at most~$ \tw(T \boxtimes K_4) \leq (1+1) \cdot 4 - 1 = 7$. 
\end{proof}

\section{Lower Bound for Geodesic Treewidth of Planar Graphs}
\label{sec:GS:GSLowerBoundPlanar}

Ueckerdt, Wood, and Yi~\cite{ImprovedPlanarGraphProductStructureTheorem} show that every planar graph has row treewidth at most $6$ by providing a partition into so-called vertical paths, i.e., paths along a BFS-tree.
In addition, Dujmovi\'{c}, Joret, Micek, Morin, Ueckerdt, and Wood~\cite{PlanarGraphsQueueNumber} give a lower bound by constructing a planar graph with row treewidth $3$. 
It is asked by different authors~\cite{ImprovedPlanarGraphProductStructureTheorem,square_graphs} to close the gap between these two bounds for planar graphs or subclasses thereof.
Interestingly, the paths in~\cite{ImprovedPlanarGraphProductStructureTheorem} not only adhere to a layering, which suffices for product structure, but they are also geodesics.
That is, \cite{ImprovedPlanarGraphProductStructureTheorem} also provides an upper bound of 6 on the geodesic treewidth and, even stronger, a partition of any planar graph into paths that certifies both upper bounds simultaneously.
Similarly, we remark that the lower-bound construction in~\cite{PlanarGraphsQueueNumber} also works for geodesic treewidth.
In this section, we attack this gap between 3 and 6 by constructing a planar graph with geodesic treewidth 5.
We consider it an interesting open question whether our construction can be adapted to also improve the lower bound on the row treewidth or, most intriguing, to a graph that lower-bounds both parameters at the same time.

\LbPlanarGraph*

\begin{figure}
    \centering
    \includegraphics{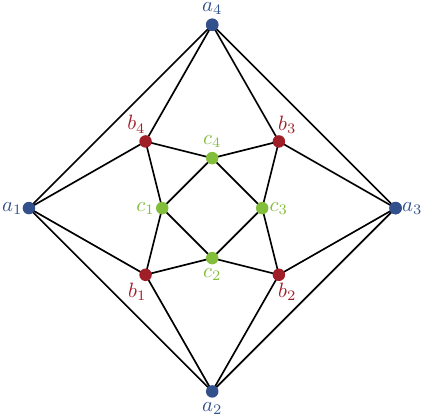}
    \caption{The cuboctahedron is known to have treewidth $5$.}
    \label{s:fig:cuboctahedron}
\end{figure}

We remark that the treewidth in \cref{thm:LbPlanarGraph} is indeed exactly 5, but we only prove the lower bound.
To do so, we construct a sequence of graphs $ G_1 \subseteq G_2 \subseteq G_3 $ and show that for every partition $ \mathcal{P} $ of $ G_3 $ into geodesics, the quotient $ G_3 / \mathcal{P} $ contains the so-called cuboctahedron (\cref{s:fig:cuboctahedron}) as a minor, a planar graph of treewidth 5.
While constructing $ G_3 $, we outline how to find the cuboctahedron step by step, and then provide a full proof 
in \cref{app:GS:GSLowerBoundPlanar}, where we also describe and analyze the construction in more detail.
For an overview, $ G_1 $ yields the blue and the green cycle in \cref{s:fig:cuboctahedron}, consisting of the $a$-vertices, respectively, $ c $-vertices.
Then $ G_2 $ adds the red $ b $-vertices to the green cycle, and finally $ G_3 $ provides the edges between $ b $- and $ c $-vertices.

During the construction, we make sure that every newly created path between already present vertices is strictly longer than the shortest path between them.
In particular, distances are preserved and we do not obtain new geodesics between previously introduced vertices.
Hence, observations concerning geodesics of some $ G_i $ also hold in subsequent supergraphs.

We start with a triangle whose vertices are called $ v_0 $, $ v_1 $, and $ v_2 $.
Note that for every partition of a triangle into geodesics, at least two of the three vertices are in distinct geodesics. 
We aim to lift this property to larger parts of our construction.
That is, we describe a graph of which we add three copies, one for each vertex of the initial triangle, and show that for every partition into geodesics, there are two copies that do not share geodesics.
In these two copies, we then find the desired cuboctahedron-minor.
\begin{figure}
    \centering
    \includegraphics{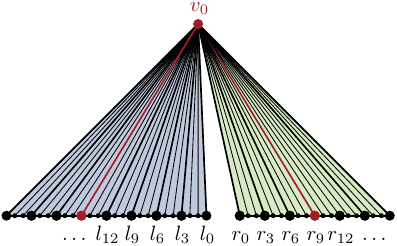}
    \caption{Two fans at $ v_0 $ in $ G_1 $, where the left fan is highlighted in blue and the right fan in green. A possible geodesic containing $v_0$ is drawn in red. Note that every geodesic has at most three vertices in any fan. 
    Thus, the thick vertices that are not red belong to pairwise distinct geodesics.}
    \label{s:fig:V0AndFans}
\end{figure}

To construct $ G_1 $ we add two fans at each $ v_i $, a \emph{left fan} and a \emph{right fan}, whose vertices are labeled as shown in \cref{s:fig:V0AndFans}.
Observe that two vertices whose distance in the path of the fan is at least 3 can only be in the same geodesic if their geodesic contains the center.
Since the fans are separated from the remainder of the graph by their center, this transfers to $ G_1 $.

\begin{restatable}{observation}{LBfans}\label{cl:lb_fans}
    For every partition of $ G_1 $ into geodesics and each of the three pairs of fans, the vertices $ r_{3k}, \ell_{3k} $ for $ k \geq 0 $ are in pairwise distinct geodesics, except for the vertices that are in the same geodesic as the center.
\end{restatable}
We use \cref{cl:lb_fans} to find cycles which serve as a base to build the cuboctahedron.
Indeed, for every partition $ \mathcal{P} $ of $ G_1 $ in geodesics, the quotient $ G_1 / \mathcal{P} $ contains six arbitrarily large cycles, at least one for each fan, provided the fans are sufficiently large.

\begin{figure}
    \begin{subfigure}[c]{0.48\textwidth}
        \centering
        \includegraphics[page=1]{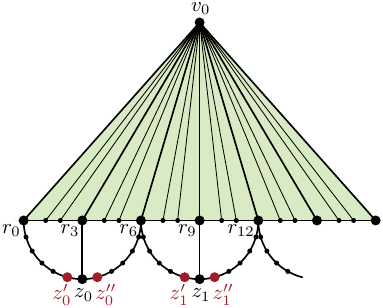}
    \end{subfigure}
    \hfill
    \begin{subfigure}[c]{0.48\textwidth}
        \centering
        \includegraphics[page=2]{figures/V0AndRightFan.pdf}
    \end{subfigure}
    \caption{Left: A right fan in $ G_2 $ with $ z_k $-attachments. Note that for all $k$, either $z'_{k}$ or $z''_{k}$ are part of a geodesic that contains no vertices in the fan. As an example, the $z_0$-attachment consists of the vertices $z_0, r_0, \dots, r_6$ and the length $6$-paths in between that contain $z_0'$ and $z_0''$.
    Right: The cycle with stacked vertices that is contained in the quotient $ G_2 / \mathcal{P} $ for every partition $ \mathcal{P} $.}
    \label{s:fig:V0AndRightFan}
\end{figure}

We extend $ G_1 $ to $ G_2 $ by adding gadgets, called \emph{$z_k$-attachments}, with three dedicated vertices $ z_k, z_k', z_k'' $ to the right fans as shown in \cref{s:fig:V0AndRightFan}.
Observe that for each $ k $, the vertices $ z'_k $ and $ z''_k $ have distance 3 to $ v_i $, and thus distance at most 4 to every vertex of their fan.
Hence, the shortest path to any vertex in the right fan of $v_i$ includes the vertex $r_{6k+3}$. 
Since a geodesic containing $r_{6k+3}$ can contain only one of $ z'_k $ and $ z''_k $, we obtain the following.

\begin{restatable}{observation}{LBzk}\label{cl:lb_zk}
    For every partition of $ G_2 $ into geodesics, each of the three right fans $ F $, and every $ k $, 
    the geodesic with $z'_{k}$ or the geodesic with $z''_{k}$ contains no vertices of~$ F $.
\end{restatable}
This allows us to extend the cycle from the quotient of $ G_1 $ by stacking vertices on the edges of the cycle as shown in \cref{s:fig:V0AndRightFan} (right).
The additional vertices are obtained from the geodesic containing $ z'_k $, respectively $ z''_k $, provided by \cref{cl:lb_zk}, whereas the other of the two vertices is not used.

\begin{figure}
    \centering
    \includegraphics{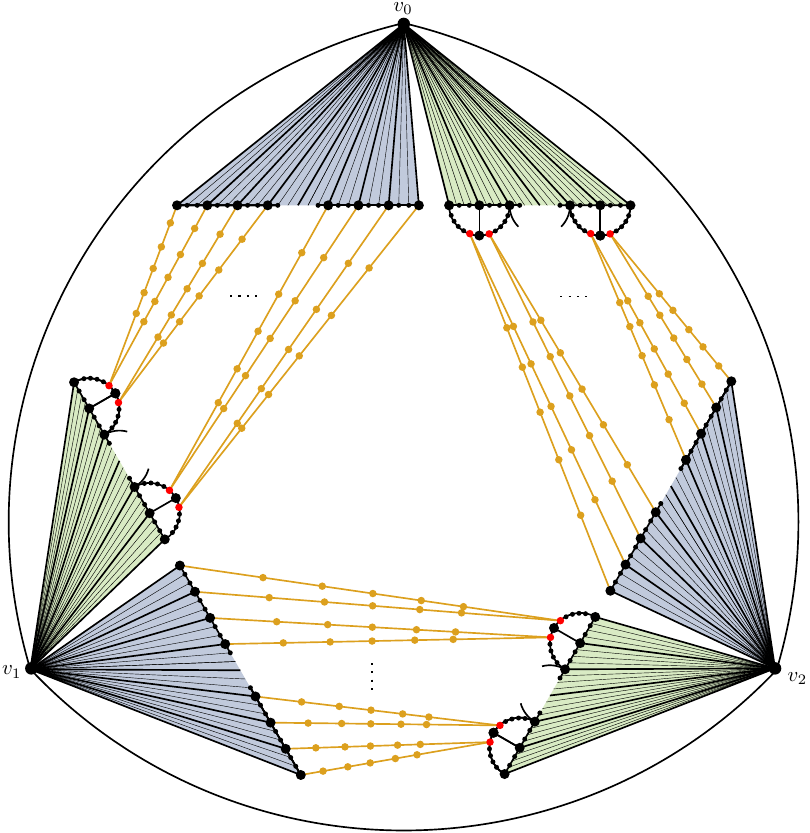}
    \caption{The planar graph $ G_3 $, which has geodesic treewidth 5. The right fans are highlighted in green and the left fans are highlighted in blue. The yellow paths are newly added in $ G_3 $.}
    \label{s:fig:LbPlanarGraph}
\end{figure}

Finally, to construct $ G_3 $ we connect the three pairs of fans with each other as shown in \cref{s:fig:LbPlanarGraph}.
To verify that distances are preserved, recall that in $ G_2 $ the distance between $ z'_k $, respectively $ z''_k $, to the center of its fan is 3, and hence at most 5 to any vertex in a fan.
Since the newly added paths in $ G_3 $ are of length 6, all paths between vertices of $ G_2 $ using a new vertex are strictly longer than the shortest path in $ G_2 $.
Moreover, for each two vertices in different copies, the shortest path contains the respective centers.

\begin{restatable}{observation}{LBcopies}\label{cl:lb_copies}
    For every partition of $ G_3 $ into geodesics, it holds that every geodesic that contains vertices of two distinct copies $ C_i $ and $ C_j $, $ i, j \in \{0, 1, 2\} $, also contains $ v_i $ and $ v_j $.
\end{restatable}
In particular, since $ v_0, v_1, v_2 $ induce a triangle and thus cannot be in a single geodesic, at least two copies do not share geodesics.
We use these two copies to find the desired cuboctahedron.

\begin{restatable}{lemma}{cuboctahedron}\label{lem:cuboctahedron}
    For every partition $ \mathcal{P} $ of $ G_3 $, the quotient $ G_3 / \mathcal{P} $ contains the cuboctahedron as a minor.
    In particular, the geodesic treewidth of $ G_3 $ is at least $ 5 $.
\end{restatable}

\begin{proof}[Proof sketch]
    \Cref{cl:lb_copies,cl:lb_fans,cl:lb_zk} give two copies such that no geodesic contains vertices of both copies as outlined during the construction.
    \Cref{s:fig:RightFanV0AndLeftFanV1,s:fig:RightFanV0AndLeftFanV1Minor} show how to find the cuboctahedron within these two copies.
    We describe this in more detail in \cref{app:GS:GSLowerBoundPlanar}.
    \begin{figure}
        \centering
        \includegraphics{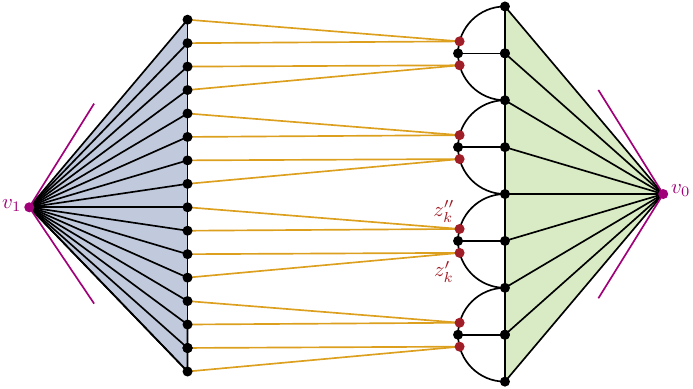}
        \caption{%
        A left sub-fan (blue) and a right sub-fan (green). 
        Subdivision vertices are not drawn. The geodesics containing $v_0$ and $v_1$ (purple) do not intersect the sub-fans besides in $v_0$ and $v_1$.
        }
        \label{s:fig:RightFanV0AndLeftFanV1}
    \end{figure}
    \begin{figure}
        \centering
        \includegraphics{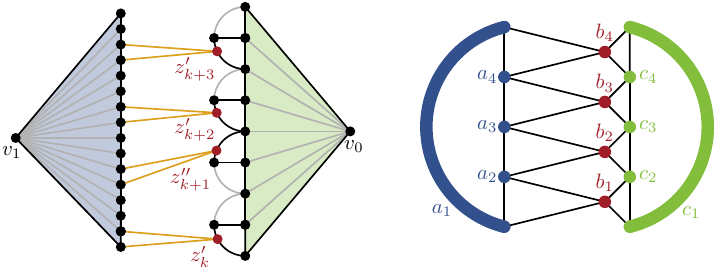}
            \caption{%
            Left: 
            Parts of the left fan of $ v_1 $ (blue) and the right fan of $v_0$ (green).
            Some vertices are omitted for readability.
            All drawn vertices of the fans and $z'_{k}, z''_{k+1}, z'_{k+2}$ and $z'_{k+3}$ belong to pairwise distinct geodesics. Gray edges are not used.
            Right: Alternative drawing of the cuboctahedron. Note that $c_1$ and $a_1$ are vertices that are drawn as half-circles to match the embedding of $ G_3 $ on the left.
            }
        \label{s:fig:RightFanV0AndLeftFanV1Minor}
    \end{figure}
\end{proof}

\section{Open Questions}

While we show that product structure does not imply geodesic structure, it is an intriguing open question whether geodesic structure implies product structure. 
To this end, we show a positive result for geodesic treewidth 1.
That is, a possible counterexample needs to be a graph class that is not minor-closed and has geodesic treewidth at least 2.
We remark that both a positive and a negative answer may have interesting consequences.
If the two notions are incomparable, i.\,e., neither implies the other, there is a chance to handle graph classes that do not admit product structure.
On the other hand, if geodesic structure is strictly stronger than product structure, there is hope to strengthen results obtained by the latter.

In light of the different behavior of geodesic treewidth and row treewidth in their computational complexity, we ask whether determining the geodesic treewidth is \FPT\ in the treewidth of the input graph.
Note that by the \NP-hardness result of Biedl, Eppstein, and Ueckerdt~\cite{biedl2023complexityembeddinggraphproducts}, row treewidth cannot be expected to admit such an \FPT-algorithm.
In addition, we remark that the concept of geodesic structure was developed to compute $p$-centered colorings for graphs of proper minor-closed graph classes and based on this to solve the subgraph-isomorphism problem on this graph class~\cite{PILIPCZUK2021111}.
We ask for further problems that can be solved more efficiently on graphs with bounded geodesic treewidth than 
in general.

Finally, we ask for the maximum geodesic treewidth and row treewidth of planar graphs, which is between 5 and 6, respectively 3 and 6, by \cref{thm:LbPlanarGraph} and~\cite{PlanarGraphsQueueNumber,ImprovedPlanarGraphProductStructureTheorem}.

\FloatBarrier

\bibliography{necessary_references}

\clearpage
\appendix

\section*{\LARGE\sffamily\bfseries\mathversion{bold}Appendix}
\vskip 1em
\crefalias{section}{appendix}

\section{Basic Definitions}
\label{app:preliminaries}
\label{sec:Introduction:Preliminaries}

Here we give some additional definitions used throughout the paper. 

\subparagraph{Basics.} 
All graphs that we consider are undirected graphs with neither loops and nor multi-edges, in particular, whenever we modify a graph we assume that loops and multi-edges are removed. 
The vertex set of a graph $ G $ is denoted by $ V(G) $ and its edge set by $ E(G) $.
For a set $ S \subseteq V(G) $, the graph that is induced by $ S $ is denoted by $ G[S] $.
\emph{Contracting} a set $ S $ of vertices is the operation of identifying all vertices of $ S $ into a single vertex which is adjacent to all neighbors of vertices in $S$. 
Note that $ S $ does not necessarily induce a connected subgraph.
We say a graph $M$ is a \emph{minor} of a graph $G$ if it can be obtained from $G$ by contracting connected subgraphs, called \emph{branch sets}, deleting vertices and deleting edges.
A graph class is called \emph{proper minor-closed} if it is closed under taking minors and is not the class of all graphs.
A \emph{layering} $\mathcal{L}=(L_1, L_2, \dots)$ of a graph $G$ is an ordered partition of its vertex set into \emph{layers} $L_1, L_2, \dots$ such that for every edge $uv \in E(G)$ with $u \in L_i, v \in L_j$ it holds that $|i-j| \leq 1$.
A \emph{partition} of a graph $ G $ is a set of induced subgraphs, also called \emph{parts}, such that their vertex sets partition the vertex set of $ G $.
The \emph{quotient} $G/\mathcal{P}$ of a graph $G$ and a partition $\mathcal{P}$ is the graph obtained from $G$ by contacting each part of $\mathcal{P}$.
The \emph{strong~product} $G_1 \boxtimes G_2$ of two graphs $G_1, G_2$ is the graph $G$ with $V(G) = V(G_1) \times V(G_2)$ and $(u_1, u_2), (v_1, v_2) \in V(G)$ are adjacent if and only if either 
$u_1 = v_1 $ and $ u_2 v_2 \in E(G_2)$, 
or $ u_1 v_1 \in E(G_1)$ and $u_2 = v_2 $,
or $ u_1 v_1 \in E(G_1)$ and $ u_2 v_2 \in E(G_2)$.
The \emph{length} of a path is the number of edges in the path.
A \emph{geodesic} is a path between two vertices $ u, v $ that is a shortest $u$-$v$ path.  
For a path $(p_1,\dots, p_k)$ in a graph $G$, a \emph{shortcut} is a path between some $p_i, p_j$ with length less than $|j-i|$. 
Observe that geodesics do not have shortcuts.
For a graph parameter $ p $ and a graph class \calG, we define $ p(\calG) = \max p(G) $, where the maximum is taken over all $ G \in \calG $.
In particular, a lower bound on $ p $ of a graph class asks only for a single certificate $ G \in \calG $.
The power set of a set $ S $ is denoted by $ 2^S $, and for a natural number $ n $ we define $ [n] = \{1, \dots, n\} $.

\subparagraph{Treewidth.} 
A \emph{tree decomposition} $(T, \mathcal{X}=\{X_1, \dots, X_k\})$ of a graph $G$ is a tree $T$ with vertices $ u_1, \dots, u_k $, where the \emph{bags} $ X_i $, $ i \in [k] $, are subsets of $ V(G)$ satisfying the following properties.
\begin{enumerate}
    \item Each vertex of $G$ is in at least one bag.
    \item For each edge $uv \in E(G)$, there is at least one bag that contains both $u$ and $v$.
    \item For each vertex $v$ of $G$, the subgraph of $T$ induced by all vertices $ u_i $ whose bags $ X_i $ contain $v$ is connected.
\end{enumerate}
The \emph{width} of a tree decomposition $(T, \mathcal{X})$ is the size of the largest bag in $\mathcal{X}$ minus 1. The \emph{treewidth} $ \tw(G) $ of a graph $G$ is the minimum width of a tree decomposition of $G$.
For a subtree $ T' \subseteq T $, let $ G[T'] $ denote the subgraph of $ G $ that is induced by all vertices in some bag of $ T' $.
We remark that the treewidth of every minor of a graph $ G $ is at most $ \tw(G) $.

A \emph{nice} tree decomposition is a special tree decomposition that only contains the following four types of bags. \emph{Leaf bags} contain only a single vertex of $G$. \emph{Introduce bags} have only a single child $c$ in $T$ and contain the vertices of $c$ and one additional vertex. \emph{Forget bags} also have only a single child $c$ in $T$ and contain the vertices of $c$ and minus one of the vertices. Lastly, \emph{join bags} have two children in $T$ and contain the same vertices as each of the two children.
For every $ n $-vertex graph of treewidth $ k $, a nice tree decomposition of width $ k $ and linear size exists and can be computed in $ O(n \cdot f(k)) $ time for some function $ f $~\cite{nice_treedecomposition}.

A graph class \calG has \emph{linear local treewidth} if there is a function $ f \in O(n) $ such that for every graph $ G \in \calG $, the $ k $-th neighborhood of every vertex has treewidth at most $ f(k) $.
We remark that linear local treewidth is implied by bounded row treewidth~\cite{PlanarGraphsQueueNumber,DUJMOVIC2017111} and by bounded geodesic treewidth, for which we give a proof for the sake of completeness.
For this, consider a graph $ G $ with geodesic treewidth $ t $ and let \calP be a corresponding partition of $ G $ into geodesics with $ \tw(G/\calP) = c $.
We show that the treewidth of the $ k $-th neighborhood of any vertex $ v $ is at most $ (2k + 1) (c + 1) $.
For this, take a tree decomposition of $ G / \calP $ of width $ c $ and for each bag, replace each vertex $ \gamma $ corresponding to a geodesic $ P_\gamma $ in $ G $ by the subpath of $ P_\gamma $ whose vertices have distance at most $ k $ to the vertex $ v $.
Observe that each geodesic contains at most $ 2k  + 1 $ such vertices.

\section{XP-Algorithm for Computing the Geodesic Treewidth}
\label{app:correctness}

In this section, we finish the algorithm outlined in \cref{sec:XPAlgorithm} by computing the set of configurations for each bag, which suffices to decide whether the geodesic treewidth is at most $ d $ by \cref{lem:RootCorrect}.
After this we conclude the section with a more detailed analysis of the runtime and a remark discussing what causes the algorithm to not be \FPT in the treewidth of the input graph.

To start with, we compute the distances between each pair of vertices of $G$ with $n$ breadth first searches.
We need these distances to later check whether certain paths are geodesics in $G$.
Since treewidth $k$ graphs have at most $k\cdot n$ edges this takes time $O((n+1)^2\cdot k^2)$.
Then, we check, whether $G$ has treewidth larger than $d$.
If not we are done.
Otherwise, we compute the set $\obs$ of forbidden minors for treewidth $d$ graphs and the set $\obssub$ of subgraphs of graphs in $\obs$.
Since $d$ is at most $k$ this takes time $O(g(k))$ for some function $g$.
In total our pre-computations take time $O(n^2 \cdot f(k))$ for some function $f$.
In the following we describe how we compute the configurations for each type of bag. 
After this we argue why the computations are sufficiently fast to obtain our desired runtime of $O(n^{2k+3}\cdot f(k))$ for some function $f$.

Next we give a more detailed description of what our algorithm outlined in \cref{sec:XPAlgorithm} does to compute the configuration for each bag type.
For this we introduce some additional notation.
For a vertex $v \in V(G)$ a partition $\mathcal{P}$ of $G$ and a minor $H$ of $G/\mathcal{P}$, we refer to the vertex $h \in V(H)$ whose branch set contains the part $P\in\mathcal{P}$ with $v \in P$ with $\minorvertex{H}{v}$.
During the algorithm we add edges to graphs.
If this would result in multi-edges or self-loops we do not add the edges but do not mention this explicitly in the following.
The descriptions of introduce, forget, and join bags are each followed by a proof of correctness.
For leaf bags the correctness is obvious.

\subparagraph{Leaf bag.}
Let $B=\{v\}$ be a leaf bag.
Clearly, there is only one possible partition $\mathcal{P}_B$ of $B$.
Namely, $\mathcal{P}_B=\{\{v\}\}$.
Recall that $\obssub$ is a set of labeled graphs.
Thus, there are multiple $B$-improper $\obssub$-minors of $G[T_B]/\mathcal{P}$ that are all isomorphic to $K_1$. 
For each of these $H\in\obssub$ we store $(H,\components)$ as a completable $B$-improper minor with $\components(h)=\{\{\{v\}\}\}$ for the unique vertex $h\in V(H)$.

\subparagraph{Introduce bag.}
Let $B$ be an introduce bag with a child $B_1 = B - \{v\}$.
We iterate over all configurations of $B_1$ to compute the configurations of $B$.
To do this, consider a configuration of $B_1$ consisting of a partition $\mathcal{P}_{1}$ with a set of labels $L_1$ and a set $M_1$ of $B_1$-improper $\obssub$-minors.
We start by computing possible partitions of $B$.
To do this, we consider three ways to add $v$ to $\mathcal{P}_1$.

\begin{description}
    \item[Adding $v$ to an existing part.] 
    We consider every way to add $v$ to a part of $\mathcal{P}_{1}$.
    For each of the resulting partitions we check whether the part $P_v$ that contains $v$ still corresponds to a geodesic in $G$.
    To do this, we first check whether $v$ is adjacent to exactly one vertex $b \in P_v$, this vertex $b$ is a border-vertex, and $b$ is not associated with an end-vertex outside of $B$ in $L_1$.
    If this is the case, let $b'$ be the other border-vertex of $B$ and $x$ be the end-vertex it is associated with in $L_1$.
    If $b'$ is not associated with such an end-vertex we set $x=b'$.
    Then, we look up the distance $d_{xb}$ between $x$ and $b$ and distance $d_{xv}$ between $x$ and $v$.
    If $d_{xb} + 1 = d_{xv}$, we know that the adding $v$ to its part resulted in a geodesic and we update the border-vertices of the partition.
    Otherwise, we discard this partition.
    
    It is left to construct the set of completable $B$-improper $\obssub$-minors.
    To do this, for each configuration of $B_1$ consisting of a partition $\mathcal{P}_{1}$ with a set of labels $L_1$ and a set $M_1$ of $B_1$-improper $\obssub$-minors and each partition $\mathcal{P}_B$ of $B$ we just computed, we iterate over all $B_1$-improper minors $(H_1,\components'_1)\in M_1$. 

    We obtain $H$ by adding edges $\minorvertex{H}{v}\minorvertex{H}{w}$ to $H_1$ for each vertex $w\in N_{G[B]}(v)$.
    To obtain $\components'$, we substitute $P_v - \{v\}$ by $P_v$ in $\components_1'(\minorvertex{H_1}{v})$ and then merge two components $C,C' \in\components'_1(\minorvertex{H_1}{v})$ if $P_v - \{v\}$ is in $C$ and a neighbor of $v$ is in a part in $C'$.

    Finally, we store $(H',\components')$ for each $H' \subseteq H$ with $H'\in\obssub$.
    
    \item[Combining two existing parts.]
    We consider every way to combine two existing parts in $\mathcal{P}_1$ by adding $v$.
    That is for each pair of parts $P_1,P_2\in\mathcal{P}_1$ we test whether we can add $v$ to both $P_1$ and $P_2$ similar to the first case.
    Then, we additionally check whether the two parts together correspond to a geodesic as well and update the labels.
    This can be done analogous to the first case.

    It is left to construct the set of completable $B$-improper $\obssub$-minors.
    To do this, we check for each $(H_1,\components'_1)\in M_1$ whether $P_1$ and $P_2$ are in the improper branch set of the same vertex $v \in V(H_1)$.
    If not we can clearly discard $(H_1,\components'_1)$.
    If they are in the same branch set we update $\components'_1(v)$ by merging the components of $P_1$ and $P_2$ and replacing $P_1$ and $P_2$ by the merged part.
    Additionally, we add edges to $H_1$ to obtain $H$ and update $\components'_1$ by possibly merging some components similar to the first case.
    Finally, we store $(H',\components')$ for each $H' \subseteq H$ with $H'\in\obssub$.

    \item[Adding a new part.] 
    We add a new part $P_v= \{v\}$ to $\mathcal{P}_B$.
    To construct the set of completable $B$-improper $\obssub$-minors, for each $(H_1,\components'_1)\in M_1$, we do two things.
    
    First, we consider every way to add $P_v$ to the branch set of a vertex $h \in V(H_1)$ with $\components'_1(h)\neq \{\{\emptyset\}\}$.
    That is, we update $\components'_1(h)$ by adding $\{P_v\}$ as a component and then possibly merge some components analogous to the cases before.
    
    Second, we consider all graphs $H \in \obssub$ such that $H_1\subset H$ and $|V(H_1)|=|V(H)|-1$ and try whether adding $P_v$ to $H_1$ results in a supergraph of $H$.
    More precisely, for the unique vertex $h \in V(H)-V(H_1)$ we set $\components'(h) = \{\{P_v\}\}$ and we add edges to $H_1$ similar to the cases before.
    If the resulting graph is a supergraph of $H$, we store $(H,\components')$ as a completable $B$-improper minor.
\end{description}

\begin{restatable}{lemma}{IntroduceBagCorrect}\label{lem:IntroduceBagCorrect}
    Given an introduce bag $B$ with a child $B_1 = B - \{v\}$ and the set of all configurations of $B_1$, our algorithm computes the set of all configurations of $B$.
\end{restatable}

\begin{proof}
    First, we consider how partitions of $G[T_B]$ and $G[T_{B_1}]$ into geodesics correspond to each other.
    Since $G[T_{B_1}]$ is an induced subgraph of $G[T_B]$, taking any partition $\mathcal{P}$ of $G[T_B]$ into geodesics, restricting it to $G[T_{B_1}]$, and then splitting any disconnected part into its connected components, results in a partition into geodesics of $G[T_{B_1}]$.
    In the other direction this corresponds to either adding $v$ as a new geodesic to the partition of $G[T_{B_1}]$, merging any two geodesics with $v$ in the middle if and only if this results in a new geodesic, or extending any geodesic with $v$ if and only if this results in a geodesic.
    Since $B_1 = B - \{v\}$ the partitions of configurations of $B$ and $B_1$ are the same except that once again $v$ can be included as its own part, $v$ can merge two parts of the partition of $B_1$ or $v$ can be added to an existing part of the partition of $B_1$. 
    In the first case our algorithm correctly computes $\mathcal{P}_B$ due to the third case in the description of the bag. In the second case our algorithm correctly computes $\mathcal{P}_B$ as it considers all possible ways to combine two existing parts by adding $v$. In the last case our algorithm correctly computes $\mathcal{P}_B$ as it considers all possible ways to add $v$ to an existing part. In the last two cases we also ensure that the underlying partition of $G[T_B]$ remains a partition into geodesics.
    
    It is left to show that our algorithm correctly computes the set $M$ of completable $B$-improper minors. We consider the three cases from the first part of the proof. 
    
    First, let $\mathcal{P}_T$ be a partition of $G[T_B]$ into geodesics such that there is a corresponding partition $\mathcal{P}_{T_1}$ of $G[T_{B_1}]$ into geodesics that differs only in $\mathcal{P}_T$ having an extra part $P_v$ containing only $v$.
    In this case the improper minors $(H, \components)$ of $G[T_B]/\mathcal{P}_T$ and improper minors $(H_1, \components_1)$ of $G[T_{B_1}]/\mathcal{P}_{T_1}$ differ only in adding $\{P_v\}$ to $\components_1(x)$ for some $x \in V(H_1)$ (or for some $x$ with $V(H)=V(H_1)\cup\{x\}$), adding the additional edges caused by $v$ and its neighbors to the minor and merging some sets in $\components_1(x)$ if $v$ connects them in $G[T_B]$. Hence, the $B$-improper minors and $B_1$-improper minors corresponding to the partition $\mathcal{P}_T$ of $G[T_B]$ and partition $\mathcal{P}_{T_1}$ of $G[T_{B_1}]$ also differ only in these aspects. Our algorithm considers all possible ways to add $\{P_v\}$ to $\components'_1(x)$ of any $B_1$-improper minor $(H_1, \components'_1)$, merges the sets in $\components'_1(x)$ such that they correspond to the connected components of $G[T_B]/\mathcal{P}_T$ and adds any resulting edges to the minor. Note that if $(H_1, \components'_1)$ is incompletable then the corresponding $B$-improper minor (where $\{P_v\}$ was added to $\components'_1(x)$) is also incompletable. If $(H,\components'_1)$ is completable then the corresponding $B$-improper minor is incompletable if and only if $\components'_1(x)$ contains $\{\emptyset\}$.
    Since our algorithm only adds $\{P_v\}$ to $\components'_1(x)$ when this is not the case, our algorithm correctly computes the set of completable $B$-improper minors.

    Second, let $\mathcal{P}_T$ be a partition of $G[T_B]$ into geodesics such that there is a corresponding partition $\mathcal{P}_{T_1}$ of $G[T_{B_1}]$ into geodesics with $\mathcal{P}_T$ being the result of merging two geodesics $P_1$ and $P_2$ of $\mathcal{P}_{T_1}$ with $v$ as the center.
    In this case the improper minors of $G[T_B]/\mathcal{P}_T$ and improper minors of $G[T_{B_1}]/\mathcal{P}_{T_1}$ differ only in replacing $P_1$ and $P_2$ with the new geodesic, adding the additional edges caused by $v$ and its neighbors to the minor, and merging some sets in $\components_1(x)$ (for $x$ with $P_1$ and $P_2$ in some subset of $\components_1(x)$) if $v$ connects them in $G[T_B]$. Hence, the $B$-improper minors and $B_1$-improper minors corresponding to the partition $\mathcal{P}_T$ of $G[T_B]$ and partition $\mathcal{P}_{T_1}$ of $G[T_{B_1}]$ also differ only in these aspects. Our algorithm replaces the parts corresponding to $P_1$ and $P_2$ by the new merged part, merges the sets in $\components'_1(x)$ such that they correspond to the connected components of $G[T_B]/\mathcal{P}_T$ and adds any resulting edges to the minor. Note that if an $B_1$-improper minor is incompletable then the corresponding $B$-improper minor (where $P_1$ and $P_2$ were merged) is also incompletable since $P_1 \neq \emptyset$ and $P_2 \neq \emptyset$. If an $B_1$-improper minor is completable then the corresponding $B$-improper minor is also completable. Therefore, our algorithm correctly computes the set of completable $B$-improper minors.

    Lastly, let $\mathcal{P}_T$ be a partition of $G[T_B]$ into geodesics and such that there is a corresponding partition $\mathcal{P}_{T_1}$ of $G[T_{B_1}]$ into geodesics that differs only in $v$ being added to a geodesic $P$ of $\mathcal{P}_{T_1}$. In this case the improper minors of $G[T_B]/\mathcal{P}_T$ and improper minors of $G[T_{B_1}]/\mathcal{P}_{T_1}$ differ only in replacing $P$ with the new extended geodesic, adding the additional edges caused by $v$ and its neighbors to the minor and merging some sets in $\components_1(x)$ (for $x$ with $P$ in some subset of $\components_1(x)$) if $v$ connects them in $G[T_B]$. Hence, the $B$-improper minors and $B_1$-improper minors corresponding to the partition $\mathcal{P}_T$ of $G[T_B]$ and partition $\mathcal{P}_{T_1}$ of $G[T_{B_1}]$ also differ only in these aspects. Our algorithm replaces the part corresponding to $P$ by the new part that includes $v$, merges the sets in $\components'_1(x)$ such that they correspond to the connected components of $G[T_B]/\mathcal{P}_T$ and adds any resulting edges to the minor. Note that an $B_1$-improper minor is completable if and only if the corresponding $B$-improper minor is also completable. Therefore, our algorithm correctly computes the set of completable $B$-improper minors.

    Thus, in all cases our algorithm computes exactly the configurations of $B$.
\end{proof}

\subparagraph{Forget bag.}
Let $B$ be a forget bag with a child $B_1 = B + \{v\}$.
Once again, we iterate over all configurations of $B_1$ to compute the configurations of $B$.
To do this, consider a configuration of $B_1$ consisting of a partition $\mathcal{P}_{1}$ with a set of labels $L_1$ and a set $M_1$ of $B_1$-improper $\obssub$-minors.

Let $P_v \in \mathcal{P}_{1}$ be the part containing $v$.
We obtain $\mathcal{P}_{B}$ by removing $v$ from $P_v$ and deleting $P_v-\{v\}$ if it is empty now.
Then, if $v$ was a border-vertex, we set the vertex next to it as the new border-vertex (recall that each part is ordered).
Additionally, if $v$ was associated with an end-vertex outside of $B$ in $L_1$ we associate the new border-vertex with this end-vertex instead.
Otherwise, we associate the new border-vertex with $v$.
Finally, we update each $B_1$-improper minor $(H_1,\components_1') \in M_1$ by setting $\components'(v)=\{\{P \cap B \colon P \in C\} \colon C \in\components_1(v)\}$ and deleting the tuple if there is a vertex $v \in V(H)$ such that $\{\emptyset\} \in \components'(v)$ and $|\components'(v)|>1$. Note that this is by definition an incompletable minor.

\begin{restatable}{lemma}{ForgetBagCorrect} \label{lem:ForgetBagCorrect}
    Given a forget bag $B$ with a child $B_1 = B + \{v\}$ and the set of all configurations of $B_1$, our algorithm computes the set of all configurations of $B$.
\end{restatable}

\begin{proof}
    We note that $G[T_B]=G[T_{B_1}]$ meaning the subgraph of $G$ corresponding to $B$ is the same as the subgraph corresponding to $B_1$. Thus, any partition of $G[T_B]$ into geodesics is a partition of $G[T_{B_1}]$ into geodesics and vice versa. This means for configurations of $B$ and $B_1$ their partitions $\mathcal{P}_B$ and $\mathcal{P}_1$ differ only in their intersection with $B$ or $B_1$. This corresponds to only a removal of $v$ from the corresponding set and removing the whole set if it is empty afterwards. 
    Note that our algorithm performs precisely this operation on partitions of configurations of $B_1$.
     
    Next we consider the $B$-improper minors in a configuration of $B$. For a partition $\mathcal{P}_T$ of $G[T_B]$ into geodesics we consider the corresponding configuration of $B_1$ and $B$ as described in the first part of this proof. Note that an improper minor of $G[T_B]/\mathcal{P}_T$ is also an improper minor of $G[T_{B_1}]/\mathcal{P}_T$ since $G[T_B]/\mathcal{P}_T=G[T_{B_1}]/\mathcal{P}_T$. Hence, similar to before, the $B$-improper minors and $B_1$-improper minors corresponding to a partition $\mathcal{P}_T$ of $G[T_B]$ and $G[T_{B_1}]$ differ only in their intersection with $B$ and $B_1$. Since $B \subseteq B_1$ intersecting the $B_1$-improper minors with $B$ (more precisely intersecting the individual sets within each branch set with $B$) corresponds exactly to the $B$-improper minors. 
    Note that if an $B_1$-improper minor is incompletable then the corresponding $B$-improper minor is also incompletable. If an $B_1$-improper minor $(H,\components'_1)$ is completable then the corresponding $B$-improper minor $(H,\components')$ is incompletable if and only if there is a vertex $v \in V(H)$ such that $\{\emptyset\} \in \components'(v)$ and $|\components'(v)|>1$. Therefore, our algorithm correctly computes the set of completable $B$-improper minors.
\end{proof}

\subparagraph{Join bag.}
Let $B$ be a join bag with two children $B_1$ and $B_2$ on the same vertex-set.
We iterate over all combinations of a configuration of $B_1$ and a configuration of $B_2$ to compute the configurations of $B$.
To do this, consider a configuration of $B_1$ consisting of a partition $\mathcal{P}_{1}$ with a set of labels $L_1$ and a set $M_1$ of $B_1$-improper $\obssub$-minors and a configuration of $B_2$ consisting of a partition $\mathcal{P}_{2}$ with a set of labels $L_2$ and a set $M_2$ of $B_2$-improper $\obssub$-minors.

First, if $\mathcal{P}_{1}\neq \mathcal{P}_{2}$ we immediately discard this combination of configurations.
Otherwise, we check whether for each part $P$ each border-vertex of $P$ is assigned in at most one of the partitions $\mathcal{P}_{1}$ or $\mathcal{P}_{2}$ an end-vertex.
If this is the case, we construct a new label set $L$ by taking the union of $L_1$ and $L_2$.
That is, each border-vertex that has an end-vertex assigned in either $L_1$ or $L_2$ is assigned this label.
Now, we check if the resulting part corresponds to a geodesic in $G$.
If not, we discard this combination of configurations.
This can be done similar to before for introduce bags by comparing the length of the resulting path with the precomputed distance.
Then, if both border-vertices of a part are assigned an end-vertex we do not store the assigned end-vertices and instead just store that there are two end-vertices associated with the part.
Doing this we obtain a partition $\mathcal{P}_B$ and a set of labels $L$.

It remains to construct the set $M$ of completable $B$-improper $\obssub$-minors.
To do so, we iterate over each pair of a $B_1$-improper minor $(H_1,\components'_1) \in M_1$ and a $B_2$-improper minor $(H_2,\components'_2) \in M_2$.
Then, we check whether for each vertex $v \in V(H_1) \cap V(H_2)$ we have that $\bigcup_{C_1\in\components'_1(v)}\bigcup_{P\in C_1} P=\bigcup_{C_2\in\components'_2(v)}\bigcup_{P\in C_2} P$ and that $\{\emptyset\}\notin \components'_1(v)\cup\components'_2(v)$.
The first part ensures that each part in $\mathcal{P}_B$ is assigned to precisely one branch set.
The second part is to discard cases in which there is a vertex $v \in V(H_1) \cap V(H_2)$ with $\components'_1(v)=\{\emptyset\}$ or $\components_2'(v)=\{\emptyset\}$, which we need to discard by the definition of completable $B$-improper minors.
Now, we construct $\components'$ for the completable $B$-improper minor $(H=H_1\cup H_2,\components'_1)$.

To do this, for each vertex $v \in V(H_1)\cap V(H_2)$ we construct an auxiliary graph $F$ with vertex set $\components'_1(v) \cup \components'_2(v)$ and an edge between two vertices $S_1,S_2\in V(F)$ if and only if there is a set $P\neq \emptyset$ with $P \in S_1\cap S_2$. For each connected component of $F$ let the union of the vertices of this component be a set in $\components'(v)$. For all other vertices $v\in V(H_1) \cup V(H_2) -(V(H_1) \cap V(H_2))$ we set $\components'(v)=\{\emptyset\}$.

Doing this for all pairs of $B_1$-improper minors and $B_2$-improper minors, we obtain the set $M$ of all completable $B$-improper $\obssub$-minors $(H,\components')$. Finally, we store $(H',\components')$ for each $H' \subseteq H$ with $H'\in\obssub$.

\begin{restatable}{lemma}{JoinBagCorrect} \label{lem:JoinBagCorrect}
    Given a join bag $B$ with two children $B_1$ and $B_2$ on the same set of vertices and the set of all configurations of $B_1$ and $B_2$, our algorithm computes the set of all configurations of $B$.
\end{restatable}

\begin{proof}
    First, consider how a partition of $G[T_B]$ corresponds to partitions of $G[T_{B_1}]$ and $G[T_{B_2}]$.
    For this, note that $G[T_B]= G[T_{B_1}] \cup G[T_{B_2}]$ with $G[B] = G[T_{B_1}] \cap G[T_{B_2}]$.
    Thus, a partition of $G[T_B]$ into geodesics corresponds to a pair of partitions of $G[T_{B_1}]$ and $G[T_{B_2}]$ into geodesics.
    In the other direction, a pair of partitions $\mathcal{P}_{T_1}$ of $G[T_{B_1}]$ and $\mathcal{P}_{T_2}$ of $G[T_{B_2}]$ do not necessarily correspond to a partition of $G[T_B]$.
    Observe that they only correspond to a partition of $G[T_B]$ if they coincide on $B$ and if taking the union of any two parts $P_1\in \mathcal{P}_{T_1}$ and $P_2 \in \mathcal{P}_{T_2}$ with $P_1 \cap P_2 \neq \emptyset$ results in a geodesic in $G[T_B]$.
    Note that our algorithm combines precisely pairs of configurations that correspond to partitions that fulfill this.

    It is left to show that the algorithm properly computes the set of completable $B$-improper minors.
    To do this consider an improper minor $(H,\components)$ of $G[T_B]/\mathcal{P}_T$ with $\mathcal{P}_T$ being a partition of $G[T_B]$ into geodesics. Let $(H_1,\components_1)$ be an improper minor of $G[T_{B_1}]/\mathcal{P}_{T_1}$ where each geodesic of $\mathcal{P}_{T_1}$ is assigned to the same branch set as the corresponding geodesic in $\mathcal{P}_{T}$ is assigned in $(H,\components)$. Note that some branch sets in $(H,\components)$ might not be assigned any geodesics that intersect $G[T_{B_1}]$ and, thus, we might have $V(H)\neq V(H')$. Let $(H_2,\components_2)$ be an analogously defined improper minor of $G[T_{B_2}]/\mathcal{P}_{T_2}$. Furthermore, let $(H_1,\components_1')$ be the corresponding $B_1$-improper minor and $(H_2,\components_2')$ the corresponding $B_2$-improper minor. Note that $H_1 \cup H_2 = H$ and that to obtain the $B$-improper minor $(H,\components')$ corresponding to $(H,\components)$ it is sufficient to compute $\components'(v)$ for each $v\in V(H_1) \cap V(H_2)$ which our algorithm does. For all other vertices $v$ it holds that $\components'(v)=\{\{\emptyset\}\}$.
    
    On the other hand, a pair of an $B_1$-improper minor $(H_1,\components_1')$ and an $B_2$-improper minor $(H_2,\components_2')$ only correspond to an $B$-improper minor if and only if each part occurring in both minors is assigned to the same branch set in $(H_1,\components_1')$ and $(H_2,\components_2')$.
    Observe that our algorithm combines only such minors, and does so by checking that $\bigcup_{C_1\in\components'_1(v)}\bigcup_{P\in C_1} P=\bigcup_{C_2\in\components'_2(v)}\bigcup_{P\in C_2} P$. 
    Note that if $(H_1,\components'_1)$ and $(H_2,\components'_2)$ fulfill this condition they correspond to the $B$-improper minor $(H_1 \cup H_2,\components')$ with $\components'(v)= \{\emptyset\}$ for each $v\in V(H_1)\cup V(H_2) - (V(H_1) \cap V(H_2))$ and depending only on $\components'_1$ and $\components'_2$ for all other vertices.
    It can be easily verified that our algorithm computes $\components'$ correctly.

    Note that if $(H_1,\components_1')$ or $(H_2,\components_2')$ is incompletable then the corresponding $B$-improper minor is also incompletable. 
    On the other hand, if $(H_1,\components_1')$ and $(H_2,\components_2')$ are completable then the corresponding $B$-improper minor is completable if and only if for all $v \in V(H_1)\cap V(H_2)$ it holds that $\{\emptyset\}\notin \components'_1(v)\cup\components'_2(v)$.
    Since our algorithm excludes these cases, it correctly computes the set of completable $B$-improper minors.
    
    Thus, in total our algorithm computes the set of all configurations of $B$.
\end{proof}

Finally, we give a more detailed analysis of the runtime than in \cref{sec:XPAlgorithm}.
As argued before each bag has $O(n^{k+1}\cdot g(k))$ distinct configurations for some function $g$ and our pre-computations take time $O(n^2 \cdot h(k))$.
For join bags we iterate over all pairs of configurations which takes time $O(n^{2k+2}\cdot f(k))$ for some function $f$.
For all other bag types we iterate over all configurations which is dominated by the runtime for join bags.
Since we have $O(n)$ bags in total and all other operations on the bags depend only on $k$, we get a runtime of $O(n^{2k+3}\cdot f(k))$ for some function $f$.

\begin{remark}\label{rem:FPT}
    The reason that our algorithm is not \FPT in the treewidth of the input graph is that we store labels for each part that associate an end-vertex outside of the bag with it.
    We need these end-vertices to check whether adding a vertex to a part (introduce bag) or merging to parts (join bag) yields a geodesic.
    Apart from this we do not use these end-vertices anywhere in the algorithm.
    Thus, if our goal is to partition the given graph $G$ into (induced) paths instead of geodesics we can omit the end-vertices and, thus, obtain configurations that are only dependent on the treewidth of $G$.
    Additionally, we then do not need to pre-compute pairwise distances of the vertices in $G$.
    In total this gives us an $O(f(k)\cdot n)$ algorithm if we drop the restriction to geodesics.
\end{remark}

\section{Computing the Geodesic Treewidth is NP-hard}
\label{app:GS:NpHard}

In this section, we present omitted proofs from \cref{sec:GS:NpHard}.

\SPGSNP*

\begin{proof}
    First, we assume that $I$ is a non-satisfiable instance. Let $\mathcal{P}$ be a $s$-$t$-partition of $G_4$ into geodesics with $S \in \mathcal{P}$ being a shortest $s$-$t$-path.
    We prove that $\tw(G_4/\mathcal{P}) \geq 3$.
    Since $I$ is non-satisfiable, there exists a clause $c \in C$ such that no literal in $c$ is in $S$.
    
    Assume $c$ is a clause of size $3$ and let $c=\{l_1,l_2,l_3\}$.
    We show that the graph $G_4/\mathcal{P}$ contains a $K_4$ minor obtained from $S$ and the geodesics in $\mathcal{P}$ containing the $l_i$.
    Since every shortest path between distinct $l_i$ and $l_j$ includes vertices in $S$, it holds that $l_i$ and $l_j$ are not part of the same geodesic in $\mathcal{P}$.
    Therefore, the vertices in $G_4/\mathcal{P}$ corresponding to the geodesic $S$ and the geodesics containing $l_1, l_2$ and $l_3$, respectively, form a $K_4$ minor.
    
    Assume now that $c=\{l_1, l_2\}$ is a clause of size $2$.
    Then, once again, we show that the graph $G_4/\mathcal{P}$ contains a $K_4$ minor.
    For this, consider the five vertices $p_1,\dots,p_5$ that are adjacent to $l_1, l_2$ and $s$ in $G_3$.
    At most two of these vertices are in the same geodesic as $l_1$ and at most two of them are in the same geodesic as $l_2$ in $G_4$.
    Without loss of generality, let $p_1$ not be part of the geodesics corresponding to $l_1$ and $l_2$.
    Recall that since $S \subseteq G_1$ it holds that $p_1 \notin S$.
    Additionally, similar to case that $c$ is a clause of size $3$, it holds that $l_1$ and $l_2$ are not part of the same geodesic.
    Therefore, the vertices in $G_4/\mathcal{P}$ corresponding to the geodesic $S$ and the geodesics containing $l_1, l_2,$ and $p_1$, respectively, are part of a $K_4$ minor.
    Thus, the graph $G_4/\mathcal{P}$ has treewidth at least $3$.

    Second, we assume that $I$ is a satisfiable instance. Let $\varphi$ be a satisfying truth assignment for $I$ and let $S$ be the corresponding $s$-$t$-path. We show that $G_4 - S$ has treewidth $1$ and thus $G_4$ has geodesic treewidth at most $2$. 
    
    It holds that for every $c \in C$ at least one literal of $c$ is in $S$. We consider the graph $G_2-S$. Let $l_1$ be a vertex in $G_2-S$ and thus $\varphi(l_1)=f$. Next, we observe that $l_1$ has degree at most $1$ in $G_2-S$. This is the case since $\varphi(l_1)=f$, so for every clause $c=\{l_1, l_2\}$ the literal $l_2$ is true and thus $l_2 \in S$. Therefore, edges introduced for clauses of size 2 do not contribute to the degree of $l_1$ in $G_2-S$. Similarly, for every clause $c=\{l_1, l_2, l_3\}$ either $l_2$ or $l_3$ is true and thus $l_2 \in S$ or $l_3 \in S$. Since $l_1$ occurs in at most one clause of length $3$ it holds that $l_1$ has degree at most $1$ in $G_2-S$. Since every vertex in $G_2-S$ has degree at most $1$, $G_2-S$ is the disjoint union of edges and singletons. Therefore, the treewidth of $G_2-S$ is $1$.
    
    Next, we lift the argument to $G_3$. The graph $G_3-S$ compared to $G_2-S$ only has the additional vertices $p_i$ for $c \in C$.
    However, since for any $c=\{l_1, l_2\} \in C$ either $l_1$ or $l_2$ is on $S$ it holds that the neighborhood $N_{G_3-S}(p_i)$ of $p_i$ in $G_3-S$ is at most a single vertex for $c \in C$ and $i \in [5]$.
    Thus, $p_i$ can be added to the tree decomposition of $G_2-S$ by adding a bag containing $p_i$ and $N_{G_3-S}(p_i)$ adjacent to a bag containing $N_{G_3-S}(p_i)$. Therefore, the treewidth of $G_3-S$ is 1. 
    
    Finally, we lift the argument to $G_4$. 
    To do this note that $G_4-S$ can be obtained from $G_3-S$ by subdividing some edges and adding some disjoint paths that are each attached to one vertex of $G_3-S$.
    Since these modifications do not increase the treewidth beyond $1$, the graph $G_4-S$ also has treewidth at most $1$. Thus, the geodesic treewidth of $G_4$ is at most $2$ and more specifically there exists an $s$-$t$-partition $\mathcal{P}$ of $G_4$ into geodesics such that $G_4/\mathcal{P}$ has treewidth at most 2. 
\end{proof}

\GSNP*

\begin{proof}
    \begin{figure}

        \centering
    \begin{subfigure}[m]{0.4\textwidth}
        \centering
        \includegraphics[page=1]{figures/hardness_quotient.pdf}
        \subcaption{}
        \label{fig:geodesic_extension}
    \end{subfigure}%
    \hfill
    \begin{subfigure}[m]{0.56\textwidth}
        \centering
        \includegraphics[page=3]{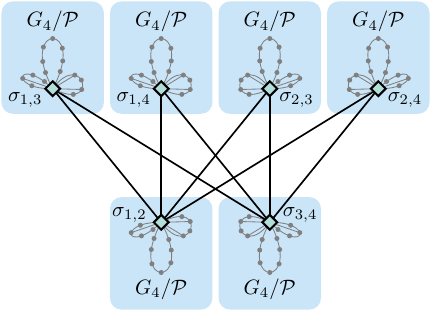}
        \subcaption{}
        \label{fig:hardness_quotient}
    \end{subfigure}%
        \caption{\subref{fig:geodesic_extension} A sketch of a graph $G$ with the geodesics $S_{i,j}$ highlighted in green, \subref{fig:hardness_quotient} the corresponding graph $G/\mathcal{P'}$. 
        The copies of $G_4/\mathcal{P}$ are highlighted in blue. For ease of presentation only a subgraph of $G_4/\mathcal{P}$ is depicted and edges are subdivided less often.}
        \label{fig:NPTD}
    \end{figure}

    Let $I$ be a satisfiable instance. \Cref{lem:SP-GSNP} gives us that $G_4/\mathcal{P}$ has treewidth $2$ for a $s$-$t$-partition $\mathcal{P}$.
    Note that the geodesics corresponding to the geodesics in $\mathcal{P}$ are also geodesics in $G$ since every path between two vertices in the same copy of $G_4$ that uses vertices outside this copy has length at least $4n$. We define a partition $\mathcal{P'}$ of $G$ into geodesics as follows. Let $\mathcal{P'}$ contain the geodesics of $\mathcal{P}$ for all copies of $G_4$. Let $S_{i,j}$ be the geodesic corresponding to the $s$-$t$-path in the copy of $G_4$ between $v_i$ and $v_j$. Finally, we extend $S_{1,2}$ by $v_1$ and $v_2$ and $S_{3,4}$ by $v_3$ and $v_4$ to get $\mathcal{P'}$ (see \cref{fig:geodesic_extension}). Let $\sigma_{i,j}$ be the vertex in $G/\mathcal{P'}$ corresponding to the, possibly extended, geodesic $S_{i,j}$.

    We claim that $G/\mathcal{P}'$ has treewidth at most $2$.
    For an example of $G/\mathcal{P}'$ see \cref{fig:hardness_quotient}.
    Additionally, deleting this subgraph yields components that are each attached to exactly one of the vertices of the $K_{2,4}$ in $G/\mathcal{P}'$.
    As each of these components, including the vertex of the $K_{2,4}$ it is attached to, is isomorphic to a subgraph of $G_4/\mathcal{P}$, it has treewidth at most $2$.
    Thus, each $2$-connected component of $G/\mathcal{P}'$ has treewidth at most $2$ and $G$ has geodesic treewidth at most 2.

    Now let $I$ be a non-satisfiable instance. Let $\mathcal{P'}$ be any partition of $G$ into geodesics.
    If no two distinct vertices $v_i$ and $v_j$, i.\,e., degree $3$ vertices of the $K_4$, are in a common geodesic in $\mathcal{P}'$, then $G/\mathcal{P'}$ contains a $K_4$ minor and, thus, has treewidth at least $3$.
    If $v_i$ and $v_j$ are in a common geodesic in $\mathcal{P'}$ let $\mathcal{P}$ be the geodesics of $\mathcal{P'}$ restricted to the copy of $G_4$ between $v_i$ and $v_j$.
    It holds that $\mathcal{P}$ is an $s$-$t$-partition into geodesics of $G_4$.
    In addition, $G_4/\mathcal{P}$ is a subgraph of $G/\mathcal{P'}$.
    Since $I$ is non-satisfiable, we know by \Cref{lem:SP-GSNP} that for any $s$-$t$-partition $P$ of $G_4$ into geodesics the graph $G_4/\mathcal{P}$ has treewidth greater than $2$.
    Thus, the geodesic treewidth of $G$ is at least $3$ if $I$ is a non-satisfiable instance.   
\end{proof}

\GSInc*

\begin{proof}
     We first show that if $G$ has geodesic treewidth $k$ then $G^{\star}$ has geodesic treewidth $k+1$. Let $\mathcal{P}$ be a partition of $G$ into geodesics such that $G/\mathcal{P}$ has treewidth $k$. The geodesics in $\mathcal{P}$ are also geodesics in $G^{\star}$ since any possible shortcut using $v$ has length at least $2|V(G)|$. We thus define a partition $\mathcal{P}^{\star}$ of $G^{\star}$ into geodesics consisting of $\mathcal{P}$ for each copy of $G$ and $\{\{x\} \mid x \in V(G^{\star})\setminus (V(G_1) \cup V(G_2) \cup V(G_3))\}$. Thus $\mathcal{P}^{\star}$ has all geodesics corresponding to $\mathcal{P}$ for the copies of $G$ and all other vertices are geodesics of length 0. We consider the graph $G'$ that consists of three disjoint copies of $G/\mathcal{P}$ joined by a universal vertex $v$. Since adding a universal vertex to a graph increases the treewidth by one, $G'$ has treewidth $k+1$. Observe that $G^{\star}/\mathcal{P}^{\star}$ is isomorphic to $G'$ with some additional subdivisions and thus $\tw(G^{\star}/\mathcal{P}^{\star}) \leq k+1$. Therefore, $G^{\star}$ has geodesic treewidth $k+1$.

     Next we show that if $G^{\star}$ has geodesic treewidth $k+1$ then $G$ has geodesic treewidth $k$. Let $\mathcal{P}^{\star}$ be a partition of $G^{\star}$ into geodesics such that $G^{\star}/\mathcal{P}^{\star}$ has treewidth $k+1$. Furthermore, let $v$ be the universal vertex in $G^{\star}$ and $P_v \in \mathcal{P}^{\star}$ be the geodesic containing $v$. Since $P_v$ is a path, at least one copy of $G$ does not intersect $P_v$.
     Assume without loss of generality that $G_1$ does not intersect $P_v$ and let $\mathcal{P}_1$ be $\mathcal{P}^{\star}$ restricted to $G_1$. Let $G_1'$ be the graph $G_1 / \mathcal{P}_1$ with an additional universal vertex. The graph $G_1'$ has treewidth $k+1$ since it is a subgraph of $G^{\star}/\mathcal{P}^{\star}$. Because a universal vertex increases the treewidth by one the graph $G_1 / \mathcal{P}_1$ has treewidth $k$ and thus $G_1$ has geodesic treewidth $k$.
\end{proof}

\section{Geodesic Treewidth 1 implies Product Structure}
\label{app:PSvsGS:doesGsimplyPs}

Here we finish the proof of the following lemma which is sketched in \cref{sec:PSvsGS:doesGsimplyPs}.

\GsImpliesPSForTreesLemma*

\begin{proof}
    For a partition \calP of some graph into geodesics, we call a $c$-layering \calL \emph{aligned} if for each geodesic $P\in \mathcal{P}$ the vertices of $P$ are embedded in consecutive layers of $\mathcal{L}$ according to their order in $P$.
    Note that only edges in geodesics are required to have their endpoints in two consecutive layers, whereas edges between distinct geodesics are allowed skip up to $ c - 1 $ layers.
    We prove the following stronger statement by induction on the number of geodesics in \calP.
    Note that since by definition, every aligned $c$-layering has only one vertex of each geodesic in each layer, \cref{cl:alignedLayering} implies that the 4-layered width of \calP is 1.
    \begin{claim}\label{cl:alignedLayering}
        For every graph $ G $ and every partition \calP of $ G $ into geodesics such that $ G / \calP $ is a tree, there is an aligned 4-layering. 
    \end{claim}

    If $|\mathcal{P}| = 1$, then there is an aligned $4$-layering of $G$ by putting each vertex into its own layer, ordered according to the single geodesic in \calP. 
    Next, consider a graph $ G $ with a partition \calP into at least two geodesics such that $ G / \calP $ is a tree.

    \NewDocumentCommand{\leaf}{}{\ensuremath{\gamma}}%
    \NewDocumentCommand{\parent}{}{\ensuremath{\delta}}%
    Let $\leaf, \parent \in V(G/\mathcal{P})$ with $\leaf$ being a leaf adjacent to $\parent$. Let $P_\leaf$ and $P_\parent$ be the geodesics corresponding to $\leaf$ and $\parent$, respectively, with $P_\leaf=(v_1, \dots, v_m)$ and $P_\parent=(w_1, \dots, w_n)$.
    Let $\mathcal{L}=(L_0, L_1, \dots)$ be an aligned $4$-layering of $G-P_\leaf$ which we obtain by induction. Note that in the following construction some layers may be assigned negative indices for ease of presentation.
    
    \begin{figure}
        \centering
        \includegraphics{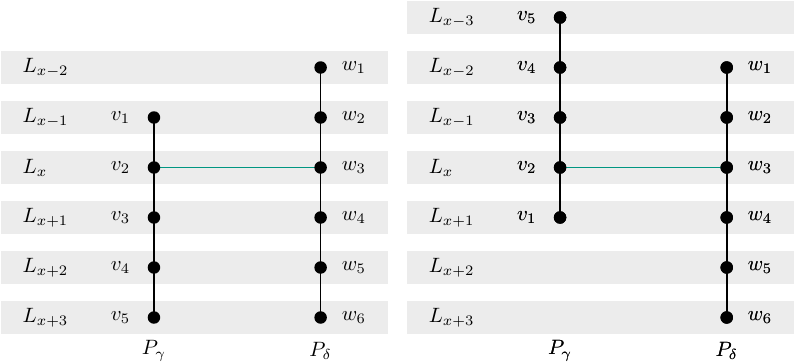}
        \caption{Layering $\mathcal{L}_1$ on the left and layering $\mathcal{L}_2$ on the right. The green edge is the alignment edge $e=v_iw_j$, where $ i = 2, j = 3 $ in this example.}
        \label{fig:AlignTwoPaths}
    \end{figure}
    
    Since $\leaf$ is adjacent to $\parent$ in $ G / \calP $, there exist $i \in [m]$ and $j \in [n]$ such that $e=v_iw_j \in E(G)$. We use this edge $e$ to align the path $P_\leaf$ and call $ e $ the \emph{alignment edge}. Let $L_x$ be the layer in $\mathcal{L}$ that contains $w_j$. We construct two aligned $c$-layerings $\mathcal{L}_1, \mathcal{L}_2$ of $G$ from $\mathcal{L}$.
    We then show that one of them indeed is an aligned 4-layering.
    To construct $\mathcal{L}_1$ we take $\mathcal{L}$ and for every $l \in [m]$ we place the vertex $v_l$ in the layer $L_{x+(l-i)}$. 
    Similarly, to construct $\mathcal{L}_2$ we take $\mathcal{L}$ and for every $l\in [m]$ we place the vertex $v_l$ in the layer $L_{x-(l-i)}$. Thus, $\mathcal{L}_2$ only differs from $\mathcal{L}_1$ in that the path $P_\leaf$ is mirrored in the layering. An example of the two possible resulting layerings can be seen in \Cref{fig:AlignTwoPaths}. For both $c$-layerings, we have that for each geodesic $P\in \mathcal{P}$ the vertices of $P$ are embedded in consecutive layers according to their order in $P$, that is, they are aligned.

    \begin{figure}
        \centering
        \begin{subfigure}{0.21\textwidth}
            \includegraphics[page=2]{figures/alignedLayerings}
            \subcaption{}
            \label{fig:NoSteepNonCrossingEdges}
        \end{subfigure}
        \hfill
        \begin{subfigure}{0.21\textwidth}
            \includegraphics[page=1]{figures/alignedLayerings}
            \subcaption{}
            \label{fig:NoNewSteepCrossingEdges}
        \end{subfigure}
        \hfill
        \begin{subfigure}{0.5\textwidth}
            \includegraphics[page=3]{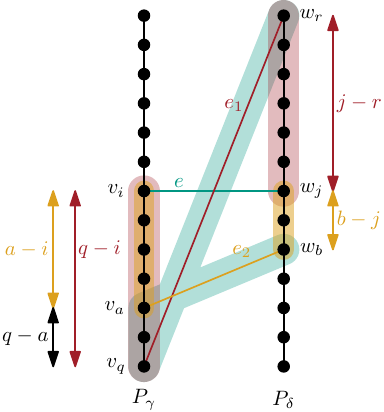}
            \subcaption{}
            \label{fig:NoLargeDiffCrossingEdges3}
        \end{subfigure}

        \caption{%
            Embeddings according to $ \calL_1 $ (where the first may also be with $ \calL_2 $).
            \textbf{\textsf{(a)}} If a non-crossing edge $e'$ is $3$-steep, then together with the alignment edge $e$, this results in a shortcut for one of the two geodesics, highlighted green.
            \textbf{\textsf{(b)}} Assuming $a \geq q$. Since $e_2$ is at most $2$-steep and $e_1$ is at least $5$-steep, the edges $e_1$ and $e_2$ result in a shortcut for $P_\parent$ that is highlighted in green. 
            \textbf{\textsf{(c)}} By the argument shown in (b), the two red subpaths have roughly the same length ($ \pm 2$), and since $ e_2 $ is at most 2-steep, the two orange subpaths have roughly the same length ($ \pm 2 $).
            Hence, the shortcut $ S $ (green) consists of a \textcolor{KITred}{red} subpath minus an \textcolor{KITorange}{orange} subpath 
            plus $ e_1, e_2 $, so has length 
            $ \textcolor{KITred}{(q-i)} - \textcolor{KITorange}{(a-i)} + \textcolor{KITgreen}{2} $.
            The subpath of $ P_\parent $ from $ w_r $ to $ w_b $ consists of a \textcolor{KITred}{red} and an \textcolor{KITorange}{orange} subpath and has length 
            $ b - r = \textcolor{KITred}{(j-r)} + \textcolor{KITorange}{(b-j)} $.
            Note that increasing the length of the orange path decreases the length of the shortcut even further, while $ P_\parent $ gets longer.
        }
    \end{figure}
    
    To prove that either $\mathcal{L}_1$ or $\mathcal{L}_2$ is an aligned $4$-layering of $G$, it remains to show that one of them is a $4$-layering. 
    For this, we only need to consider edges between $ P_\leaf $ and $ P_\parent $.
    For a layering $\mathcal{L'}$ let a \emph{$k$-steep} edge be an edge with endpoints in layers $L_a, L_b$ of $\mathcal{L'}$ and $|a - b| = k$. 
    We also call $ k $ the \emph{steepness} of the edge in this case.
    For the layerings $\mathcal{L}_1$ and $\mathcal{L}_2$, let a \emph{crossing} edge be an edge with endpoints in layers $L_a, L_b$ of $\mathcal{L}_1$, respectively $\mathcal{L}_2$, and $a < x < b$, i.\,e., edges that cross the alignment edge $e$ in the embedding shown in \Cref{fig:AlignTwoPaths}. 
    Note that non-crossing edges $e'$ are at most $2$-steep since otherwise $e$ and $e'$ result in a shortcut for one of the two geodesics $P_\leaf$ or $P_\parent$ as shown in \Cref{fig:NoSteepNonCrossingEdges}. 
    That is, it is only left to argue for crossing edges.

    Assume now that there is a crossing edge $e_1=v_qw_r$ that is at least 5-steep in $\mathcal{L}_1$ and consider a crossing edge $e_2= v_aw_b$ of $ \calL_2 $.
    We show that $e_2$ is at most 4-steep in $\mathcal{L}_2$ and thus $\mathcal{L}_2$ is an aligned $4$-layering. 
    Note that if an edge is crossing in $\mathcal{L}_1$ then it is non-crossing in $\mathcal{L}_2$ and vice versa. 
    In particular, $ e_1 $ is non-crossing in $ \calL_2 $ and $ e_2 $ is non-crossing in $ \calL_1 $.
    Without loss of generality, let $a \geq i$ and $b \geq j$, meaning that $e_2$ is below the alignment edge $e$ in the embedding of $\mathcal{L}_1$ shown in \Cref{fig:AlignTwoPaths}, and let $q > i$ and $r < j$, i.\,e., $ e_1 $ crosses $ e $ from top right to bottom left in \cref{fig:NoNewSteepCrossingEdges}.

    We first show that all edges that are crossing in $ \calL_2 $ are close to the alignment edge $ e $.
    More precisely, we show for the crossing edge $ e_2 = v_a w_b $ that $ a < q $, i.\,e., in the geodesic $ P_\leaf $ the vertex $ v_a $ is closer to $ v_i $ than the endpoint $ v_q $ of $ e_1 $, see \cref{fig:NoLargeDiffCrossingEdges3}. 
    For this, assume the opposite, i.\,e., that $ a \geq q $.
    Recall that $e_2$ is non-crossing in $\mathcal{L}_1$ and, therefore, at most $2$-steep in $\mathcal{L}_1$. Thus, the edges $e_1$ and $e_2$ result in a shortcut of $ P_\parent $ between $ w_r $ and $ w_b $ as shown in \Cref{fig:NoNewSteepCrossingEdges}.
    Indeed, the edge $ e_1 $ is at least 5-steep, whereas $ e_2 $ is at most 2-steep, so the $ w_r $-$ w_b $-subpath of $ P_\parent $ contains at least $ 5 - 2 = 3 $ vertices more than the $ v_a $-$ v_q $-subpath of $ P_\leaf $.
    Together with $ e_1 $ and $ e_2 $, the path $ (w_r, v_q, \dots, v_a, w_b) $ is shorter by at least $ 3 - 2 = 1 $ than the geodesic between $ w_r $ and $ w_b $ obtained from $ P_\parent $, a contradiction.

    That is, $ v_a $ is between $ v_i $ and $ v_q $ in $ P_\leaf $ and we aim to show that $ e_2 $ is at most 4-steep in $ \calL_2 $.
    First, since $ e_1 = v_q w_r $ is non-crossing in $ \calL_2 $, it is at most 2-steep and, thus, its endpoints have roughly the same distance to $ e = v_i w_j $, up to an offset of 2.
    More precisely, we use that $q-i\leq (j-r)+2$, where $ q - i $, respectively $ r - j $, is the distance between the endpoints of $ e_1 $ and $ e $ in the two geodesics, recall also \cref{fig:NoSteepNonCrossingEdges} with $ e' = e_1 $.
    In $ \calL_1 $, this translates to a symmetry of $ e_1 $ with respect to $ e $ as $ q-i \approx j - r $ as shown in \cref{fig:NoLargeDiffCrossingEdges3} (red).
    
    Now, consider the geodesic $ P $ from $ w_r $ via $ w_j $ to $ w_b $, which is a subpath of $ P_\parent $ and has length $ p = b - r = (b-j) + (j-r) $.
    We construct another path $ P' = (w_r, v_q, \dots, v_a, w_b) $ using edges $e_1$ and $e_2$ and the path $P_\leaf$ from $ v_q $ to $ v_a $. An example for such a path $P'$ with all relevant distances is given in \cref{fig:NoLargeDiffCrossingEdges3}. 
    Note that $P'$ has length $p' = q-a+2 = (q-i) - (a-i) + 2 $.
    Again, since $ P $ is a geodesic, $ P' $ is at least as long as $ P $, so we have 
    $ 0 \geq p - p' = (b-j) + (j-r) - (q-i) + (a-i) - 2 \geq (b-j) + (a-i) - 4 $ using the symmetry of $ e_1 $.
    Observe that $ (b-j) + (a-i) $ is the steepness of $ e_2 = v_a w_b $ in $\calL_2$ since $ v_i $ and $ w_j $ are in the same layer.
    That is, $ e_2 $ is at most $ 4 $-steep in $\calL_2$, as required, so $ \calL_2 $ is an aligned 4-layering.
\end{proof}

\section{Lower Bound for Geodesic Treewidth of Planar Graphs}
\label{app:GS:GSLowerBoundPlanar}

In this section, we give a full proof of \cref{thm:LbPlanarGraph}, which is sketched in \cref{sec:GS:GSLowerBoundPlanar}.

\LbPlanarGraph*

We remark that the treewidth in \cref{thm:LbPlanarGraph} is indeed exactly 5, but we only prove the lower bound.
To do so, we construct a sequence of graphs $ G_1 \subseteq G_2 \subseteq G_3 $ and show that for every partition $ \mathcal{P} $ of $ G_3 $ into geodesics, the quotient $ G_3 / \mathcal{P} $ contains the so-called cuboctahedron (\cref{fig:cuboctahedron}) as a minor, a planar graph of treewidth 5.
While constructing $ G_3 $, we outline how to find the cuboctahedron step by step, and then provide a full proof at the end of the section in \cref{lem:cuboctahedron} based on our observations we make during the construction.
For an overview, $ G_1 $ yields the blue and the green cycle in \cref{fig:cuboctahedron}, consisting of the $a$-vertices, respectively, $ c $-vertices.
Then $ G_2 $ adds the red $ b $-vertices to the green cycle, and finally $ G_3 $ provides the edges between $ b $- and $ c $-vertices.

\begin{figure}[htb]
    \centering
    \includegraphics{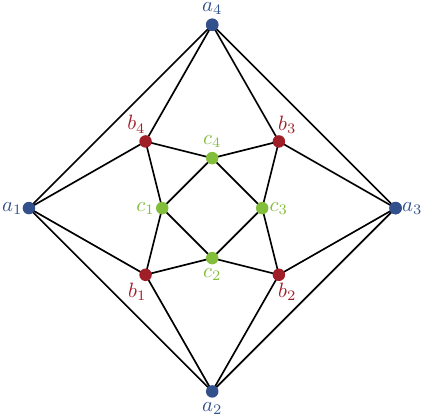}
    \caption{The cuboctahedron is known to have treewidth $5$.}
    \label{fig:Treewidth5GraphGDash}
    \label{fig:cuboctahedron}
\end{figure}

During the construction, we make sure that every newly created path between already present vertices is strictly longer than the shortest path between them.
In particular, distances are preserved and we do not obtain new geodesics between previously introduced vertices.
Hence, observations concerning geodesics of some $ G_i $ also hold in subsequent supergraphs.

We start with a triangle whose vertices are called $ v_0 $, $ v_1 $, and $ v_2 $.
Note that for every partition of a triangle into geodesics, at least two of the three vertices are in distinct geodesics. 
We aim to lift this property to larger parts of our construction.
That is, we describe a graph of which we add three copies, one for each vertex of the initial triangle, and show that for every partition into geodesics, there are two copies that do not share geodesics.
In these two copies, we then find the desired cuboctahedron-minor.

\begin{figure}
    \centering
    \includegraphics{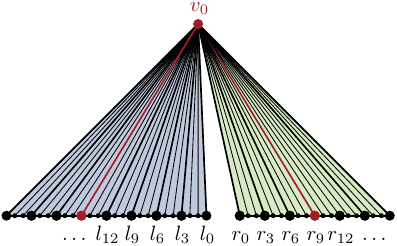}
    \caption{Two fans at $ v_0 $ in $ G_1 $, where the left fan is highlighted in blue and the right fan in green. A possible geodesic containing $v_0$ is drawn in red. Note that every geodesic has at most three vertices in any fan. 
    Thus, the thick vertices that are not red belong to pairwise distinct geodesics.}
    \label{fig:V0AndFans}
\end{figure}

To construct $ G_1 $ we add two fans at each $ v_i $ as follows, where $ v_i $ is the center.
We describe the two fans at $ v_0 $, the other two copies are inserted accordingly.
The first fan is called the \emph{right fan at $ v_0 $} and consists of vertices $ r_0, r_1, \dots $ forming a path in this order, and additionally edges between $ v_0 $, called the center, and each vertex of the path.
Similarly, the \emph{left fan at $v_0$} consists of a path $ \ell_0, \ell_1, \dots $ and a star with center $ v_0 $.
For the sake of completeness, let us note that the path in the right fan consists of 133 vertices, and the path in the left fan of 265 vertices.
However, for now we always assume that the fans are sufficiently large for our arguments and discuss the exact sizes at the end of the proof in \cref{lem:cuboctahedron}.
An example of $v_0$ and the adjacent fans is shown in \cref{fig:V0AndFans}.
Observe that the distance between any two vertices in a fan is at most 2.
Hence, within a fan, two vertices whose distance in the path is at least 3 can only be in the same geodesic if their geodesic contains the center.
Since the fans are separated from the remainder of the graph by their center, this transfers to $ G_1 $.

\LBfans*
In \cref{lem:cuboctahedron}, we use \cref{cl:lb_fans} to find cycles which serve as a base to build the cuboctahedron.
Indeed, for every partition $ \mathcal{P} $ of $ G_1 $ in geodesics, the quotient $ G_1 / \mathcal{P} $ contains six arbitrarily large cycles, at least one for each fan, provided the fans are sufficiently large.

\begin{figure}
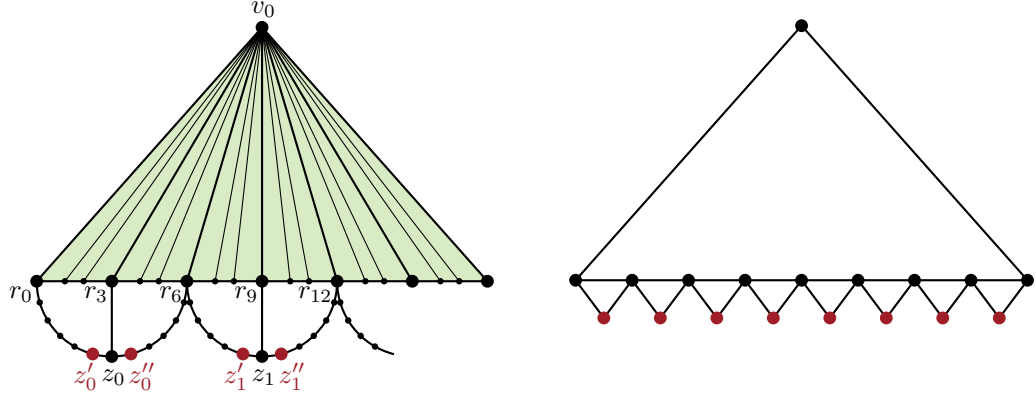

    \begin{subfigure}[c]{0.48\textwidth}
        \centering
        \includegraphics[page=1]{figures/V0AndRightFan.pdf}
    \end{subfigure}
    \hfill
    \begin{subfigure}[c]{0.48\textwidth}
        \centering
        \includegraphics[page=2]{figures/V0AndRightFan.pdf}
    \end{subfigure}
    \caption{Left: A right fan in $ G_2 $ with $ z_k $-attachments. Note that for all $k$, either $z'_{k}$ or $z''_{k}$ are part of a geodesic that contains no vertices in the fan. As an example the $z_0$-attachment consists of the vertices $z_0, r_0, \dots, r_6$ and the length $6$-paths in between that contain $z_0'$ and $z_0''$.
    Right: The cycle with stacked vertices that is contained in the quotient $ G_2 / \mathcal{P} $ for every partition $ \mathcal{P} $.}
    \label{fig:V0AndRightFan}
\end{figure}

Having $ G_1$, we extend it to $ G_2 $ by adding vertices adjacent to the right fan of $v_i$ for each $i \in \{0, 1, 2\} $.
Again, we only describe the construction for $v_0$.
First, we introduce vertices $z_0, z_1, \dots $ where each vertex $z_k$ is adjacent to $r_{6k+3}$.
In addition, each vertex $z_k$ is connected by paths of length $6$ to the vertices $r_{6k}$ and $r_{6(k+1)}$. 
We call the subgraph consisting of $ z_k, r_{6k}, \dots, r_{6(k+1)} $ and corresponding length-6 paths the \emph{$z_k$-attachment}.
An example of $v_0$ and its right fan with some $z_k $-attachments is shown in \cref{fig:V0AndRightFan}. 
Note that this introduces new paths between vertices in $ G_1 $.
However, each such path has length at least 7, which is strictly greater than maximum distance between vertices in $ G_1 $.
In particular, distances are preserved, as required.

We denote the neighbors of $ z_k $ by $ z'_k $ and $ z''_k $, where $ z'_k $ lies on the length-$6$ path to $r_{6k}$, and $ z''_k $ on the length-6 path to $r_{6(k+1)}$.
Observe that for each $ k $, the vertices $ z'_k $ and $ z''_k $ have distance 3 to $ v_0 $, and thus distance at most 4 to every vertex of the fan.
Hence, the shortest path to any vertex in the right fan of $v_0$ includes the vertex $r_{6k+3}$. 
Since a geodesic containing $r_{6k+3}$ can contain only one of $ z'_k $ and $ z''_k $, we obtain the following.

\LBzk*
This allows us to extend the cycle from the quotient of $ G_1 $ by stacking vertices on the edges of the cycle as shown in \cref{fig:V0AndRightFan} (right).
The additional vertices are obtained from the geodesic containing $ z'_k $, respectively $ z''_k $, provided by \cref{cl:lb_zk}, whereas the other of the two vertices is not used.

\begin{figure}
    \centering
    \includegraphics{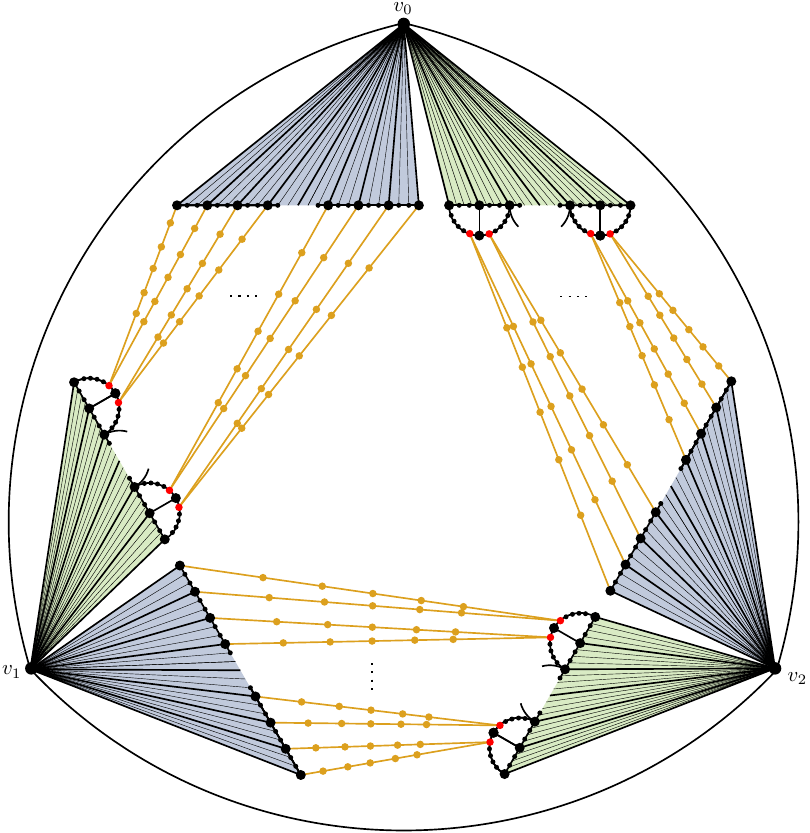}
    \caption{The planar graph $ G_3 $, which has geodesic treewidth 5. The right fans are highlighted in green and the left fans are highlighted in blue. The yellow paths are newly added in $ G_3 $.}
    \label{fig:LbPlanarGraph}
\end{figure}

Finally, to construct $ G_3 $ we connect the three pairs of fans with each other, which we illustrate in \cref{fig:LbPlanarGraph}.
For this, recall that each of the three vertices $ v_0, v_1, v_2 $ of the initial triangle has its own copy of a left fan and a right fan including the vertices and edges added in $ G_2 $.
We denote these copies by $ C_0 $, $ C_1 $, and $ C_2 $.
For $ i \in \{0, 1, 2\} $ (indices taken modulo~3) and every $ k $, we connect $z'_{k}$ of copy $ C_i $ to the vertices $\ell_{12k}$ and $\ell_{12k+3}$ of copy $C_{i+1}$ using two paths of length $6$.
We also connect $z''_{k}$ of copy $C_i$ to $l_{12k+6}$ and $l_{12k+9}$ of copy $C_{i+1}$ using another two paths of length $6$. 
We say a $ z_k $-attachment of some right fan is connected to a vertex of a different fan if $ z'_k $ or $ z''_k $ are connected to it via one of the newly added length-6 paths.
Observe that $ G_3 $ indeed admits a planar embedding such that the copies of $ r_0 $ and $ \ell_0 $ share a common face.

To verify that distances are preserved, recall that in $ G_2 $ the distance between $ z'_k $, respectively $ z''_k $, to the center of its fan is 3, and hence at most 5 to any vertex in a fan.
Since the newly added paths in $ G_3 $ are of length 6, all paths between vertices of $ G_2 $ using a new vertex are strictly longer than the shortest path in $ G_2 $.
Moreover, for each two vertices in different copies, the shortest path contains the respective centers.

\LBcopies*
In particular, since the three vertices $ v_0, v_1, v_2 $ induce a triangle and thus cannot be in a single geodesic, at least two copies do not share geodesics.
We use these two copies to find the desired cuboctahedron.

\cuboctahedron*

\begin{proof}
    To find a cuboctahedron-minor in the quotient $ G_3 / \mathcal{P} $, we aim to identify suitable geodesics in $ \mathcal{P} $ that correspond to vertices of the cuboctahedron, which we refer to by colors and names as shown in \cref{fig:cuboctahedron}.
    Without loss of generality, the vertices $ v_0 $ and $ v_1 $ of the initial triangle are in distinct geodesics, denoted by $ P_0 $ and $ P_1 $, respectively.
    By \cref{cl:lb_copies}, no geodesic contains vertices of both copies $ C_0 $ and $ C_1 $.
    Thus, vertices we find in one of the copies for the cuboctahedron do not coincide with those in the other copy.
    
    \begin{figure}
        \centering
        \includegraphics{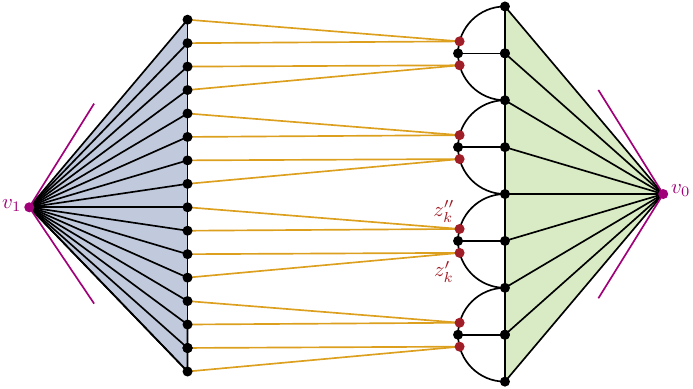}
        \caption{In green the right sub-fan of $v_0$ and in blue the left sub-fan of $v_1$. Subdivision vertices are not drawn for better readability. In purple the geodesics containing $v_0$ and $v_1$ that do not intersect the sub-fans besides in $v_0$ and $v_1$.
        }
        \label{fig:RightFanV0AndLeftFanV1}
    \end{figure}
    
    We start with finding the green cycle consisting of $ c $-vertices in the right fan at $ v_0 $ and the blue cycle consisting of $ a $-vertices in the left fan of $ v_1 $.
    Crucially, we need to align the two cycles such that they can be connected to a cuboctahedron.
    Since the distance between any two vertices of a fan is at most 2, $ P_0 $ contains $ v_0 $ and at most two additional vertices of the right fan at $ v_0 $.
    Note that these two vertices are contained in at most four $ z_k $-attachments, up to two adjacent ones per vertex.
    Similarly, $ P_1 $ contains $ v_1 $ and at most two further vertices in the left fan at $ v_1 $, each of which is connected to at most one $ z_k $-attachment. 
    That is, there are at most six $ z_k $-attachments that contain a vertex in $ P_0 $ or are connected to a vertex in $ P_1 $, and removing them splits the fans into at most five sub-fans.
    Hence, assuming the fans are sufficiently large such that we have $ 6 + 5 \cdot 3 + 1 = 22 $ $ z_k$-attachments, we find sub-fans with four consecutive $ z_k $-attachments that neither contain a vertex of $ P_0 $ nor are connected to one of $ P_1 $.
    We refer to \cref{fig:RightFanV0AndLeftFanV1} for an illustration.

    \begin{figure}
        \centering
        \includegraphics{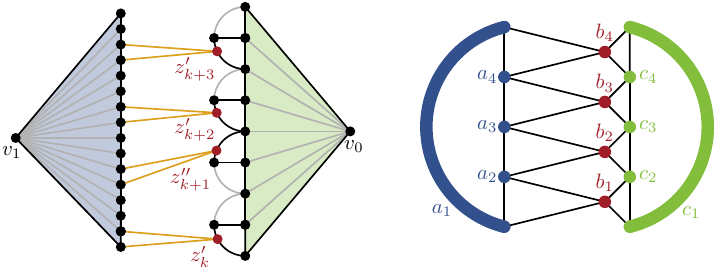}
        \caption{Left: The left sub-fan of $v_1$ (blue) and the right sub-fan of $v_0$ (green). Subdivision vertices are not drawn and only every third vertex of the fans is drawn for better readability. All the drawn vertices of the fans and $z'_{k}, z''_{k+1}, z'_{k+2}$ and $z'_{k+3}$ belong to pairwise distinct geodesics. The edges and paths in light gray are not needed for the cuboctahedron.
        Right: Alternative drawing of the cuboctahedron that is known to have treewidth $5$. Note that $c_1$ and $a_1$ are vertices that are drawn as half-circles to match the embedding of $ G_3 $ on the left.
        }
    \label{fig:RightFanV0AndLeftFanV1Minor}
\end{figure}
    
    Within these sub-fans, we now find a cuboctahedron-minor, which is shown in \cref{fig:RightFanV0AndLeftFanV1Minor} on the right with an embedding that matches the subgraphs of $ G_3 $ shown in \cref{fig:RightFanV0AndLeftFanV1,fig:RightFanV0AndLeftFanV1Minor}.
    First, by \cref{cl:lb_fans}, the vertices $ r_{3k} $ of the right fan are in pairwise distinct geodesics, and they are distinct of $ P_0 $ by choice of the considered sub-fan.
    Hence, these geodesics together with $ P_0 $ form a cycle-minor.
    By \cref{cl:lb_zk}, for each $ k $, one of $ z_k' $ and $ z_k'' $ does not share a geodesic with the right fan.
    Thus, their geodesics provide additional vertices for our desired minor as shown in \cref{fig:RightFanV0AndLeftFanV1Minor}, where they serve as $b$-vertices.
    Finally, the left fan at $ v_1 $ provides the $ a $-cycle that is connected to the $ b $-vertices via the paths introduced when constructing $ G_3 $ from $ G_2 $.
    
    To conclude, we remark that in order to have 22 $ z_k $-attachments, it suffices that the path of each right fan consists of $ 22 \cdot 6 + 1 = 133 $ vertices, and the path of each left fan has $ 22 \cdot 12 + 1 = 265 $ vertices. 
\end{proof}
\end{document}